\documentclass[aop,MSNbibl,seceqn,dvips]{arximspdf}
\usepackage{mathbh}

\doi{10.1214/11-AOP647}
\volume{40}
\issue{3}
\pubyear{2012}
\firstpage{979}
\lastpage{1008}

\makeatletter

\newcommand{\N}{\mathbb{N}}
\newcommand{\Z}{\mathbb{Z}}

\newcommand{\E}{\mathbb{E}}
\renewcommand{\P}{\mathbb{P}}
\newcommand{\var}{\operatorname{Var}}
\newcommand{\Cov}{\operatorname{Cov}}

\renewcommand{\epsilon}{\varepsilon}

\newcommand{\cG}{{\mathcal{G}}}
\newcommand{\given}{|}
\newcommand{\one}{\mathbh{1}}
\newcommand{\deq}{\stackrel{\triangle}{=}}
\newcommand{\K}{\mathcal{K}}
\newcommand{\cH}{\mathcal{H}}
\newcommand{\GC}{{\mathcal{C}_1}} 
\newcommand{\tGC}{{\tilde{\mathcal{C}}_1}}
\newcommand{\TC}[1][\mathcal{C}_1]{#1^{(2)}} 
\newcommand{\Po}{\operatorname{Po}}
\newcommand{\Bin}{\operatorname{Bin}}
\newcommand{\Geom}{\operatorname{Geom}}
\newcommand{\tmixu}{\tilde{t}_{\mathrm{MIX}}}
\newcommand{\tmix}{t_{\mathrm{MIX}}}
\newcommand{\diam}{\operatorname{diam}}
\newcommand{\dist}{\operatorname{dist}}
\newcommand{\eqref}[1]{(\ref{#1})}

\newtheorem{maintheorem}{Theorem}
\newtheorem{theorem}{Theorem}[section]
\newtheorem{lemma}[theorem]{Lemma}
\newtheorem{claim}[theorem]{Claim}
\newtheorem{proposition}[theorem]{Proposition}
\makeatother

\begin{document}
\begin{frontmatter}

\title{Mixing time of near-critical random graphs}
\runtitle{Mixing time of near-critical random graphs}

\begin{aug}
\author[A]{\fnms{Jian} \snm{Ding}\ead[label=e1]{jding@stat.berkeley.edu}},
\author[B]{\fnms{Eyal} \snm{Lubetzky}\corref{}\ead[label=e2]{eyal@microsoft.com}}
and
\author[B]{\fnms{Yuval} \snm{Peres}\ead[label=e3]{peres@microsoft.com}}

\runauthor{J. Ding, E. Lubetzky and Y. Peres}
\affiliation{University of California, Berkeley, Microsoft Research and~Microsoft~Research}
\address[A]{J. Ding\\
Department of Statistics\\
University of California, Berkeley\\
Berkeley, California 94720\\
USA\\
\printead{e1}} 
\address[B]{E. Lubetzky\\
Y. Peres\\
Microsoft Research\\
One Microsoft Way\\
Redmond, Washington 98052-6399\\
USA\\
\printead{e2}\\
\hphantom{E-mail: }\printead*{e3}}
\end{aug}

\received{\smonth{10} \syear{2009}}
\revised{\smonth{1} \syear{2011}}

%
\begin{abstract}
Let $\mathcal{C}_1$ be the largest component of the Erd\H{o}s--R\'
{e}nyi random
graph $\mathcal{G}(n,p)$. The mixing time of random walk on $\mathcal
{C}_1$ in the
strictly supercritical regime, $p=c/n$ with fixed $c>1$, was shown to
have order $\log^2 n$ by Fountoulakis and Reed, and independently by
Benjamini, Kozma and Wormald. In the critical window,
$p=(1+\varepsilon)/n$ where $\lambda=\varepsilon^3 n$ is bounded, Nachmias
and Peres proved that the mixing time on $\mathcal{C}_1$ is of order $n$.
However, it was unclear how to interpolate between these results, and
estimate the mixing time as the giant component emerges from the
critical window. Indeed, even the asymptotics of the diameter of
$\mathcal{C}_1$
in this regime were only recently obtained by Riordan and Wormald, as
well as the present authors and Kim.

In this paper, we show that for $p=(1+\varepsilon)/n$ with
$\lambda=\varepsilon^3 n\to\infty$ and $\lambda=o(n)$, the mixing
time on
$\mathcal{C}_1$ is with high probability of order $(n/\lambda)\log^2
\lambda$. In
addition, we show that this is the order of the largest mixing time
over all components, both in the slightly supercritical and in the
slightly subcritical regime [i.e., $p=(1-\varepsilon)/n$ with $\lambda
$ as
above].
\end{abstract}

%
\begin{keyword}[class=AMS]
\kwd{05C80}
\kwd{05C81}
\kwd{60J10}.
\end{keyword}
\begin{keyword}
\kwd{Random graphs}
\kwd{random walk}
\kwd{mixing time}.
\end{keyword}

\end{frontmatter}

\section{Introduction}

There is a rich interplay between geometric properties of a graph and
the behavior of a random walk on it (see, e.g., \cite{AF}).
A particularly important parameter is the mixing time, which measures
the rate of convergence to stationarity.
In this paper, we focus on random walks on the classical Erd\H{o}s--R\'
{e}nyi random graph $\cG(n,p)$.

The geometry of $\cG(n,p)$ has been studied extensively since its
introduction in 1959 by Erd\H{o}s and R\'{e}nyi~\cite{ER59}.
A well-known phenomenon exhibited by this model, typical in
second-order phase transitions of mean-field models, is the \emph
{double jump}: For $p=c/n$ with $c$ fixed, the largest component $\GC$
has size $O(\log n)$ with high probability (w.h.p.), when $c<1$, it is
w.h.p.\ linear in $n$ for $c>1$, and for $c=1$ its size has order
$n^{2/3}$ (the latter was proved by Bollob\'{a}s~\cite{Bollobas84} and
{\L}uczak~\cite{Luczak90}). Bollob\'{a}s discovered that the critical
behavior extends throughout $p=(1\pm\epsilon)/n$ for $\epsilon=
O(n^{-1/3})$, a regime known as the \emph{critical window}.

Only in recent years were the tools of Markov chain analysis and the
understanding of the random graph sufficiently developed
to enable estimating mixing times on $\GC$.
Fountoulakis and Reed \cite{FR2} showed that, in the strictly
supercritical regime ($p=c/n$ with fixed $c>1$),
the mixing time of random walk on $\GC$
w.h.p.\ has order $\log^2 n$. Their proof exploited fairly simple
geometric properties of $\cG(n,p)$, while the key to their analysis
was a refined bound \cite{FR1} on the mixing time of a general Markov chain.
The same result was obtained independently by Benjamini, Kozma and
Wormald \cite{BKW}.
There, the main innovation was a decomposition theorem for the giant component.
However, the methods of these two papers do not yield the right order
of the mixing time when $c$ is allowed to tend to $1$.

Nachmias and Peres \cite{NP} proved that throughout the critical
window the mixing time on $\GC$ is of order $n$.
The proof there used branching process arguments, which were effective
since the critical $\GC$ is close to a tree.

It was unclear how to interpolate between these results, and estimate
the mixing time as the giant component emerges from the critical
window, since the methods used for the supercritical and the critical
case were so different.
The focus of this paper is primarily on the emerging supercritical
regime, where $p=(1+\epsilon)/n$ with $\epsilon^3 n \to\infty$ and
$\epsilon=o(1)$. In this regime, the largest component is
significantly larger than the others, yet its size is still sublinear.
Understanding the geometry of $\GC$ in this regime has been
challenging: Indeed, even the asymptotics of its diameter were only
recently obtained by Riordan and Wormald~\cite{RW}, as well as in
\cite{DKLP2}.

Our main result determines the order of the mixing time throughout the
emerging supercritical regime (see Section \ref{subsec:mixing} for
a formal definition of mixing time).

\begin{maintheorem}[(Supercritical regime)]\label{mainthm-super}
Let $\GC$ be the largest component of $\cG(n,p)$ for $p = \frac{1 +
\epsilon}{n}$, where $\epsilon\to0$ and $\epsilon^3
n\to\infty$. With high probability, the mixing time of the lazy
random walk on $\GC$ is of order $\epsilon^{-3} \log^2 (\epsilon^3 n)$.
\end{maintheorem}

While the second largest component $\mathcal{C}_2$ has a mixing time
of smaller order (it is
w.h.p.\ a tree, and given that event, it is a uniform tree on its
vertices and as such has $\tmix\asymp|\mathcal{C}_2|^{3/2}$ (see,
e.g., \cite{NP}), that is $\tmix\asymp \epsilon^{-3} \log
^{3/2}(\epsilon^3 n)$ as $|\mathcal{C}_2| \asymp\epsilon^{-2} \log
(\epsilon^3 n)$ w.h.p.), it turns out that w.h.p.\ there exists an
even smaller component, whose mixing time is of the same order as on
$\GC$.
This is captured by our second theorem, which also handles the
subcritical regime.

\begin{maintheorem}[(Controlling all components)]\label{mainthm-sub}
Let $G \sim\cG(n,p)$ for $p = (1 \pm\epsilon)/n$, where $\epsilon
\to0$ and $\epsilon^3
n\to\infty$. Let $\mathcal{C}^\star$ be the component of $G$
that\vadjust{\goodbreak}
maximizes the mixing time of the lazy random walk on it, denoted by
$\tmix^\star$.
Then with high probability, $\tmix^\star$ has order $\epsilon^{-3}
\log^2 (\epsilon^3 n)$. This also holds when maximizing only over
tree components.
\end{maintheorem}

In the area of random walk on random graphs, the following two regimes
have been analyzed extensively.
\begin{itemize}
\item The supercritical regime, where $\tmix\asymp(\diam)^2$ with
$\diam$ denoting the intrinsic diameter in the percolation cluster.
Besides $\cG(n, \frac{c}{n})$ for $c>1$, this also holds in the torus
$\Z_n^d$ by \cite{BM} and \cite{MR}.
\item The critical regime on a high dimensional torus, where $\tmix
\asymp(\diam)^3$.
As mentioned above, for critical percolation on the complete graph,
this was shown in \cite{NP}. For high dimensional tori, this is a
consequence of \cite{HvdH}.
\end{itemize}
To the best of our knowledge, our result is the first interpolation for
the mixing time between these two different powers of the diameter.


\section{Preliminaries}\label{sec:prelim}

\subsection{Cores and kernels} The \emph{$k$-core} of a graph $G$,
denoted by $G^{(k)}$, is the maximum subgraph $H \subset G$ where every
vertex has degree at least $k$. It is well known (and easy to see) that
this subgraph is unique, and can be obtained by repeatedly deleting any
vertex whose degree is smaller than $k$ (at an arbitrary order).

We call a path $\mathcal{P} = v_0, v_1, \ldots, v_k$ for $k > 1$
(i.e., a sequence of vertices with $v_i v_{i+1}$ an edge for each $i$)
a \emph{2-path} if and only if $v_i$ has
degree $2$ for all $i = 1, \ldots, k-1$ (while the endpoints $v_0,v_k$
may have degree larger than $2$, and possibly $v_0=v_k$).

The \emph{kernel} $\K$ of $G$ is obtained by taking its $2$-core $\TC
[G]$ minus its disjoint cycles, then repeatedly contracting all 2-paths
(replacing each by a single edge). Note that, by definition, the degree
of every vertex in $\K$ is at least $3$.

\subsection{Structure of the supercritical giant component}
The key to our analysis of the random walk on the giant component $\GC
$ is the following result from our companion paper \cite{DKLP1}.
This theorem completely characterizes the structure of $\GC$, by
reducing it to a tractable contiguous model $\tGC$.
\begin{theorem}[\cite{DKLP1}]\label{thm-struct-gen}
Let $\GC$ be the largest component of $\cG(n,p)$ for $p = \frac{1 +
\epsilon}{n}$, where $\epsilon^3 n\to\infty$ and $\epsilon\to0$.
Let $\mu<1$ denote the conjugate of $1+\epsilon$, that is,
$\mu\mathrm{e}^{-\mu} = (1+\epsilon) \mathrm{e}^{-(1+\epsilon)}$.
Then $\GC$ is contiguous to the following model $\tGC$:
\begin{enumerate}
\item\hypertarget{item-struct-gen-degrees}{} Let $\Lambda\sim\mathcal{N}
(1+\epsilon- \mu, \frac1{\epsilon n} )$ and assign
i.i.d.\  variables $D_u \sim\operatorname{Poisson}(\Lambda)$ ($u \in[n]$)
to the vertices, conditioned that $\sum D_u \one_{\{D_u\geq3\}}$ is even.
Let
\[
N_k = \#\{u \dvtx D_u = k\}\quad  \mbox{and}\quad N= \sum_{k\geq3}N_k.
\]
Select a random multigraph $\K$ on $N$ vertices, uniformly among all
multigraphs with $N_k$ vertices of degree $k$ for $k\geq3$.
\item\hypertarget{item-struct-gen-edges}{} Replace the edges of $\K$ by
paths of lengths i.i.d.\ $\Geom(1-\mu)$. 
%
\item\hypertarget{item-struct-gen-bushes}{} Attach an independent $
\operatorname{Poisson}(\mu)$--Galton--Watson tree to each vertex.
\end{enumerate}
That is, $\P(\tGC\in\mathcal{A}) \to0$ implies $\P(\GC\in
\mathcal{A}) \to0$
for any set of graphs $\mathcal{A}$ that is closed under graph-isomorphism.
\end{theorem}

In the above, a Poisson($\mu$)--Galton--Watson tree is the family tree
of a~Galton--Watson branching process with offspring distribution
$\operatorname{Poisson}(\mu)$. We will use the abbreviation PGW($\mu
$)-tree for this object. A multigraph is the generalization of a simple
graph permitting multiple edges and loops.

Note that conditioning on $\sum D_u \one_{\{D_u\geq3\}}$ being even
does not pose a problem, as one can easily use rejection sampling.
The 3 steps in the description of $\tGC$ correspond to constructing
its kernel $\K$ (Step~\hyperlink{item-struct-gen-degrees}{1}), expanding $\K$
into the 2-core $\TC[\tGC]$ (Step~\hyperlink{item-struct-gen-edges}{2}), and
finally attaching trees to it to obtain $\tGC$ (Step \hyperlink{item-struct-gen-bushes}{3}).

Further observe that $N_k \asymp\epsilon^k n$ for any fixed $k\geq
2$, and so in the special case where $\epsilon= o(n^{-1/4})$
w.h.p.\ we have $D_u \in\{0,1,2,3\}$ for all $u\in[n]$, and the
kernel $\K$ is simply a uniform $3$-regular multigraph.

Combining the above description of the giant component with standard
tools in the study of random graphs with given degree-sequences, one
can easily read off useful geometric properties of the kernel. This is
demonstrated by the following lemma of \cite{DKLP1}, for which we
require a few definitions:
For a~vertex $v$ in $G$ let $d_G(v)$ denote its degree and for a~subset
of vertices $S$ let
\[
d_G(S) \deq\sum_{v\in S} d_G(v)
\]
denote the sum of the degrees of its vertices (also referred to as the
\emph{volume} of $S$ in $G$). The \emph{isoperimetric number} of a
graph $G$ is defined to be
\[
i(G) \deq\min \biggl\{ \frac{e(S,S^c)}{d_G(S)} \dvtx S \subset V(G) , d_G(S)
\leq e(G) \biggr\} ,
\]
where $e(S,T)$ denotes the number of edges between $S$ and $T$ while
$e(G)$ is the total number of edges in $G$.
\begin{lemma}[(\cite{DKLP1}, Lemma 3.5)]\label{lem-kernel-expander}
Let $\K$ be the kernel of the largest component $\GC$ of $\cG(n,p)$
for $p = \frac{1 + \epsilon}{n}$, where $\epsilon^3 n\to\infty$
and $\epsilon\to0$. Then w.h.p.,
\[
|\K|= \bigl(\tfrac43+o(1) \bigr)\epsilon^3 n ,\qquad  e(\K) = \bigl(2+o(1)\bigr)\epsilon^3 n ,
\]
and $i(\K) \geq\alpha$ for some absolute constant $\alpha>0$.
\end{lemma}

\subsection{Notions of mixing of the random walk}\label{subsec:mixing}
For any two distributions~$\varphi,\psi$ on~$V$, the
\textit{total-variation distance} of $\varphi$ and $\psi$ is defined as
\[
\|\varphi-\psi\|_\mathrm{TV} \deq\sup_{S \subset V} |\varphi(S) -
\psi(S) | = \frac{1}{2}\sum_{v\in V} |\varphi(v)-\psi(v)| .\vadjust{\goodbreak}
\]
Let $(S_t)$ denote the lazy random walk on $G$, that is, the Markov chain
which at each step holds its position with probability $\frac12$ and
otherwise moves to a uniformly chosen neighbor. This is an aperiodic
and irreducible Markov chain,
whose stationary distribution $\pi$ is given by
\[
\pi(x) = d_G(x) / 2|E| .
\]
We next define two notions of measuring the distance of an ergodic
Markov chain $(S_t)$, defined on a state-set $V$, from its stationary
distribution $\pi$.

Let $0 < \delta< 1$. The (worst-case) \emph{total-variation mixing
time} of $(S_t)$ with parameter $\delta$, denoted by $\tmix(\delta
)$, is defined to be
\[
\tmix(\delta) \deq\min\Bigl\{t \dvtx \max_{v \in V} \| \P_v(S_t \in\cdot
)- \pi\|_\mathrm{TV} \leq\delta \Bigr\} ,
\]
where $\P_v$ denotes the probability given that $S_0=v$.

The \emph{Ces\`{a}ro mixing time} (also known as the approximate
uniform mixing
time) of $(S_t)$ with parameter $\delta$, denoted by $\tmixu(\delta
)$, is defined as
\[
\tmixu(\delta) = \min\Biggl\{t \dvtx\max_{v \in V}
\Biggl\|\pi-\frac{1}t\sum_{i=0}^{t-1}\P_v(S_i\in
\cdot) \Biggr\|_\mathrm{TV} \leq\delta \Biggr\} .
\]

When discussing the order of the mixing-time it is customary to choose
$\delta=\frac14$, in which case we will use the abbreviations
$\tmix= \tmix(\frac14)$ and $\tmixu= \tmixu(\frac14)$.

By results of \cite{ALW} and \cite{LW95} (see also \cite{LW}), the
mixing time and the Ces\`{a}ro mixing time have the same
order for lazy reversible Markov chains (i.e., discrete-time chains
whose holding probability in each state is at least $\frac12$), as
formulated by the following theorem.
\begin{theorem}\label{thm-Cesaro-mixing-time}
Every lazy reversible Markov chain satisfies
\[
c_1 \tmixu\bigl(\tfrac14\bigr) \leq\tmix\bigl(\tfrac14\bigr) \leq c_2 \tmixu\bigl(\tfrac14\bigr)
\]
for some absolute constants $c_1,c_2 >0$.
\end{theorem}
\begin{pf}
The first inequality is straightforward and does not require
laziness or reversibility. We include its proof for completeness.
Notice that
\begin{eqnarray*}
\Biggl\|\pi-\frac{1}t\sum_{i=0}^{t-1}\P_v(S_i\in
\cdot) \Biggr\|_\mathrm{TV} &\leq&\frac{1}{8} +
\frac1t\sum_{i={t/8}}^{t-1} \|\pi-\P_v(S_i\in
\cdot) \|_\mathrm{TV} \\
&\leq&\frac{1}{8} + \|\pi- \P_v(S_{t/8} \in
\cdot)\|_{\mathrm{TV}} ,
\end{eqnarray*}
where we used the fact that $\|\pi- \P_v(S_{t} \in\cdot)\|$ is
decreasing in $t$. Taking $t = 8 \tmix(\frac{1}{8})$, we obtain
that $\tmixu(\frac{1}{4}) \leq8 \tmix(\frac{1}{8})$ and conclude
the proof of the first inequality using the well-known fact
that $\tmix(\frac{1}{8}) \leq4 \tmix(\frac{1}{4})$.

The second inequality of the theorem is significantly more involved:
By combining \cite{LW95}, Theorem 5.4, (for a stronger version, see
\cite{LW}, Theorem 4.22) and \cite{ALW}, Theorem C, it follows that
the order of the Ces\`{a}ro mixing time can be bounded by that of the
mixing time
for the corresponding continuous-time Markov chain. Now, using a
well-known fact that the mixing time for the lazy Markov chain and
the continuous-time chain have the same order (see, e.g.,
\cite{LPW}, Theorem 20.3), the proof is concluded.
\end{pf}


Let $\Gamma$ be a stopping rule (a randomized stopping time) for
$(S_t)$. That is, $\Gamma\dvtx G \times\Omega\to\N$ for some
probability space $\Omega$, such that $\Gamma(\cdot,\omega)$ is a
stopping time for every $\omega\in\Omega$.
Let $\sigma^\Gamma\deq\P_\sigma(S_\Gamma\in\cdot)$ when $\sigma
$ is a distribution on $V$.

Let $\sigma,\nu$ be two distributions on $V$. Note that there is
always a stopping rule $\Gamma$ such that $\sigma^\Gamma=\nu$,
for example, draw a vertex $z$ according to $\nu$ and stop when reaching $z$.
The \emph{access time} from $\sigma$ to $\nu$, denoted by $H(\sigma
,\nu)$, is the minimum expected number of steps over all such stopping rules:
\[
H(\sigma,\nu) \deq\min_{\Gamma\dvtx \sigma^\Gamma= \nu} \E\Gamma.
\]
It is easy to verify that $H(\sigma,\nu)=0$ iff $\sigma=\nu$ and
that $H(\cdot,\cdot)$ satisfies the triangle-inequality, however it
is not necessarily symmetric.

The \emph{approximate forget time} of $G$ with parameter $0 < \delta<
1$ is defined by
%
\begin{equation}
\label{eq-def-approx-forget}
\mathcal{F}_{\delta} = \min_\varphi\max_{\sigma} \min_{\nu\dvtx \|
\nu- \varphi\|_\mathrm{TV}\leq\delta} H(\sigma, \nu) .
\end{equation}
Combining Theorem 3.2 and Corollary 5.4 in \cite{LW}, one
immediately obtains that the approximate forget time and the Ces\`{a}ro
mixing time have the same order, as stated in the following theorem.
\begin{theorem}\label{thm-forget-mix}
Every reversible Markov chain satisfies
\[
c_1 \mathcal{F}_{1/4} \leq\tmixu\bigl(\tfrac14\bigr) \leq c_2 \mathcal{F}_{1/4}
\]
for some absolute constants $c_1,c_2 >0$.
\end{theorem}

\subsection{Conductance and mixing}
Let $P = (p_{x, y})_{x,y}$ be the transition kernel of an irreducible,
reversible and aperiodic Markov chain on $\Omega$ with stationary
distribution $\pi$. For $S \subset\Omega$, define the
\textit{conductance} of the set $S$ to be
\[
\Phi(S) \deq\frac{\sum_{x\in S, y\notin S} \pi(x)p_{x,y}}{\pi
(S) \pi(\Omega\setminus S)} .
\]
We define $\Phi$, the conductance of the chain, by $\Phi\deq\min\{
\Phi(S) \dvtx \pi(S) \leq\frac12\}$
(In the special case of a lazy random walk on a connected regular
graph, this quantity is similar to the isoperimetric number of the
graph, defined earlier).
A well-known result of Jerrum and Sinclair \cite{JS} states that
$\tmix$ is of order at most $\Phi^{-2} \log\pi^{-1}_{\min}$, where
$\pi_{\min} = \min_{x \in\Omega}\pi(x)$. This bound was
fine-tuned by Lov\'{a}sz and Kannan \cite{LK} to exploit settings
where the conductance of the average set $S$ plays a dominant role
(rather than the worst set).
For our upper bound of the mixing time on the random walk on the
2-core, we will use an enhanced version of the latter bound (namely,
Theorem \ref{thm-Fountoulakis-Reed}) due to Fountoulakis and Reed
\cite{FR1}.

\subsection{Edge set notations}
Throughout the paper, we will use the following notations, which will
be handy when moving between the kernel and 2-core.

For $S \subset G$, let $E_G(S)$
denote the set of edges in the induced subgraph of $G$ on $S$, and
let $\partial_G S$ denote the edges between $S$ and its complement
$S^c \deq V(G)\setminus S$.
Let
\[
\bar{E}_G(S) \deq E_G(S) \cup\partial_G(S)
\]
and
define $e_G(S)\deq|E_G(S)|$.
We omit the subscript $G$ whenever its identity is made clear from the context.

If $\K$ is the kernel in the model $\tGC$ and $\cH$ is its $2$-core,
let
\[
E_\cH^\star\dvtx 2^{E(\K)}\to2^{E(\cH)}
\]
be the operator which takes a subset of edges $T\subset E(\K)$ and
outputs the
edges lying on their corresponding $2$-paths in $\cH$.
For $S \subset V(\K)$, we let
\[
E_\cH^\star(S) \deq E_\cH^\star(E_\K(S) ) ,\qquad
\bar{E}_\cH^\star(S) \deq E_\cH^\star(\bar{E}_\K(S) ) .
\]

\section{Random walk on the 2-core}\label{sec:rw-2-core}

In this section, we analyze the properties of the random walk on the
$2$-core $\TC[\tGC]$.

\subsection{Mixing time of the 2-core}\label{sec:two-core-mixing}
By the definition of our new model $\tGC$, we can study the 2-core
$\TC$ via the
well-known configuration model (see, e.g., \cite{Bollobas2} for
further details on this method). To simplify the notation, we let~$\cH
$ denote the 2-core of $\tGC$
throughout this section.

The main goal of the subsection
is to establish the mixing time of the lazy random walk on $\cH$, as
stated by the
following theorem.

\begin{theorem}\label{thm-core-mixing}
With high probability, the lazy random walk on $\cH$ has
a~Ces\`{a}ro mixing time $\tmixu$ of order $\epsilon^{-2}
\log^2 (\epsilon^3 n)$. Consequently, w.h.p.\ it also satisfies
$\tmix\asymp\epsilon^{-2}\log^2(\epsilon^3 n)$.
\end{theorem}

We will use a result of Fountoulakis and Reed \cite{FR2}, which bounds
the mixing time in terms of the isoperimetric profile of the graph
(measuring the expansion of sets of various volumes).\vadjust{\goodbreak} As a first step
in obtaining this data for the supercritical 2-core $\cH$,
the next lemma will show that a small subset of the kernel,
$S\subset\K$, cannot have too many edges in
$\bar{E}_{\cH}(S)$.

\begin{lemma}\label{lem-kernel-size-volume}
For $v\in\K$, define
\[
\mathfrak{C}_{v, K} \deq\{S \ni v\dvtx |S| = K \mbox{ and } S \mbox{ is a
connected subgraph of } \K\} .
\]
The following holds w.h.p.\ for every $v\in\K$, integer $K$ and $S
\in
\mathfrak{C}_{v, K}$:
\begin{enumerate}[(2)]
\item[(1)]\hypertarget{item-lem-kernel-1}{} $|\mathfrak{C}_{v, K}| \leq\exp[
5(K \vee\log(\epsilon^3 n)) ]$.
\item[(2)]\hypertarget{item-lem-kernel-2}{} $d_\K(S) \leq30 (K \vee\log
(\epsilon^3 n))$.
\end{enumerate}
\end{lemma}

\begin{pf}
By definition, $\Lambda= (2+o(1)) \epsilon$ w.h.p., thus standard
concentration arguments imply that
the following holds w.h.p.:
%
\begin{equation}\label{eq-N-k}\qquad
N_3= \biggl(\frac{4}{3}+o(1) \biggr)\epsilon^3 n \quad \mbox{and}\quad  N_k \leq\frac
{(3\epsilon)^{k}\log(1/\epsilon)}{k!}n\qquad
\mbox{for } k\geq4 .
\end{equation}
Assume that the above indeed holds, and notice that the lemma trivially
holds when $K \geq
\epsilon^3 n$. We may therefore further assume that $K \leq\epsilon
^3 n$.

Consider the following exploration process, starting from the vertex
$v$. Initialize $S$ to be $\{v\}$, and mark $v_1 = v$.
At time $i \geq1$, we explore the neighborhood of $v_i$ (unless $|S| <
i$), and for each its neighbors that does not already belong to $S$,
we toss a fair coin to decide whether or not to insert it to $S$. Newly
inserted vertices are labeled according to the order of their arrival;
that is, if $|S|=k$ prior to the insertion, we give the new vertex the
label $v_{k+1}$. Finally, if $|S|<i$ at time $i$ then we stop the
exploration process.

Let $X_i$ denote the degree of the vertex $v_i$ in the above defined
process. In order to stochastically dominate
$X_i$ from above, observe that the worst case occurs when each of the
vertices in $v_1,\ldots,v_{i-1}$ has
degree $3$. With this observation in mind, let $A$ be a set consisting
of $N_3 - K$ vertices of degree 3 and
$N_k$ vertices $k$ (for $k\geq4$). Sample a vertex proportional to the
degree from $A$ and let $Y$ denote its
degree. Clearly, $X_i \preceq Y_i$, where $Y_i$ are independent
variables distributed as $Y$, and so
%
\begin{equation}\label{eq-d(S)-stoc-dom}
d_\K(S) \preceq\sum_{i=1}^K Y_i .
\end{equation}
By the definition of our exploration process,
\[
|\mathfrak{C}_{v, K}| \preceq\sum_{\ell_1 + \cdots+\ell_K =
K}\prod_{i=1}^{K} \pmatrix{Y_i\cr\ell_i} .
\]
We can now deduce that
%
\begin{equation}\label{eq-bound-T-v-K}\qquad
\E|\mathfrak{C}_{v, K}| \leq\E \left[ \sum_{\ell_1 + \cdots+\ell_K =
K} \prod_{i=1}^{K}
\pmatrix{Y_i\cr\ell_i} \right] = \sum_{\ell_1 + \cdots+\ell_K = K} \prod
_{i=1}^{K} \E
\left[\pmatrix{Y_i\cr\ell_i} \right] .
\end{equation}
For all $i \geq4$, we have
\[
\P(Y = i) \leq27 i \frac{(3\epsilon)^{i-3}\log(1/\epsilon)}{i!} =
27 \frac{(3 \epsilon)^{ i -3 }\log(1/\epsilon) }{(i -1)!}
\]
and therefore, for sufficiently large $n$ (recall that $\epsilon= o(1)$),
\[
\E \left[\pmatrix{Y\cr k} \right] \leq\pmatrix{3\cr k} + \sum_{i\geq4} \pmatrix{i\cr k}
\cdot27\frac{(3\epsilon)^{i-3}\log(1/\epsilon) }{(i-1)!} \leq
\frac{7}{k!}\qquad  \mbox{for all }k .
\]
Altogether,
%
\begin{equation}\label{eq-E-T-v-K}
\E|\mathfrak{C}_{v, K}| \leq7^K \sum_{\ell_1 + \cdots+\ell_K =
K}\prod_{i=1}^K \frac1{\ell_i !} .
\end{equation}
The next simple claim will provide a bound on the sum in the last expression.
\begin{claim}\label{claim-algebra}
$\!\!\!$The function $f(n) = \sum_{\ell_1 + \cdots+\ell_n = n} \prod
_{k=1}^{n} \frac{1}{\ell_k !}$ satisfies \mbox{$f(n) \leq
\mathrm{e}^n$}.
\end{claim}

\begin{pf}
The proof is by induction. For $n=1$, the claim trivially holds.
Assuming the hypothesis is valid for $n\leq m$, we get
\[
f(m+1) = \sum_{k=0}^{m+1} \frac{1}{k!} f(m - k) \leq\sum
_{k=0}^{m+1}\frac{\mathrm{e}^{m-k}}{k!} \leq\mathrm{e}^{m} \sum
_{k=0}^{m+1} \frac{1}{k!} \leq\mathrm{e}^{m+1} ,
\]
as required.
\end{pf}

Plugging the above estimate into \eqref{eq-E-T-v-K}, we conclude that
$\E|\mathfrak{C}_{v, K}| \leq(7
\mathrm{e})^K$. Now, Markov's inequality, together with a union bound
over all the vertices in the
kernel $\K$ yield the Part \hyperlink{item-lem-kernel-1}{(1)} of the lemma.

For Part \hyperlink{item-lem-kernel-2}{(2)}, notice that for any sufficiently
large $n$,
\[
\E\mathrm{e}^Y \leq\mathrm{e}^3 + \sum_{i\geq4} \mathrm{e}^i 27
i \frac{(3\epsilon)^{i-3}\log(1/\epsilon)}{i!} \leq25 ,
\]
therefore, \eqref{eq-d(S)-stoc-dom} gives that
\[
\P \bigl(d_\K(S) \geq30 \bigl(K \vee\log(\epsilon^3 n) \bigr) \bigr)\leq\exp\bigl[-5 \bigl(K
\vee\log(\epsilon^3 n) \bigr) \bigr] .
\]
At this point, the proof is concluded by a union bound over $\mathfrak
{C}_{v, K}$ for all $v \in\K$ and $K \leq\epsilon^3 n$, using the
upper bound we have already derived for $|\mathfrak{C}_{v, K}|$ in the
Part \hyperlink{item-lem-kernel-1}{(1)} of the lemma.
\end{pf}

\begin{lemma}\label{lem-conductance-1}
Let $\mathcal L\subset E(\K)$ be the set of loops in the kernel.
With high probability, every subset of vertices $S \subset\K$ forming
a connected subgraph of~$\K$
satisfies
$|\bar{E}^\star_{\cH}(S)| \leq (100/\epsilon) (|S| \vee\log
(\epsilon^3 n) )$,
and every subset $T$ of $\frac1{20}\epsilon^3 n$ edges in $\K$ satisfies
$|E_{\cH}^\star(T)\cup E_{\cH}^\star(\mathcal{L})| \leq\frac
34\epsilon^2 n$.
\end{lemma}

\begin{pf}
Assume that the events given in Parts \hyperlink{item-lem-kernel-1}{(1)}, \hyperlink{item-lem-kernel-2}{(2)} of Lemma \ref{lem-kernel-size-volume} hold.
Further note that, by definition of the model $\tGC$, a standard
application of CLT yields that w.h.p.\
\[
|\K|=\bigl(\tfrac43+o(1)\bigr)\epsilon^3n ,\qquad  e(\cH) = \bigl(2+o(1)\bigr) \epsilon^2 n ,\qquad
e(\K) = \bigl(2+o(1)\bigr) \epsilon^3 n .
\]
By Part \hyperlink{item-lem-kernel-2}{(2)} of that lemma,
$d_\K(S) \leq30 (|S| \vee\log(\epsilon^3 n) )$ holds
simultaneously for every connected set $S$,
hence there are at most this many edges in $\bar{E}_\K(S)$.

Let $S\subset\K$ be a connected set of size $|S|=s$, and let
\[
K = K(s)= s \vee\log(\epsilon^3 n) .
\]
Recalling our definition of the graph
$\cH$, we deduce that
\[
|\bar{E}^\star_{\cH}(S) | \preceq\sum_{i = 1}^{30 K} Z_i ,
\]
where $Z_i$ are i.i.d.\ Geometric random variables with mean $\frac
{1}{1 - \mu}$.
It is well known that the moment-generating function of such variables
is given by
\[
\E( \mathrm{e}^{t Z_1}) = \frac{(1 - \mu)\mathrm{e}^t}{ 1 - \mu
\mathrm{e}^{t}} .
\]
Setting $t\!=\!\epsilon/2$ and recalling that $\mu\!=\!1\,{-}\,(1\,{+}\,o(1))
\epsilon$, we get that \mbox{$\E
(\mathrm{e}^{({\epsilon}/{2}) Z_1})\!\leq\!\mathrm{e}$} for
sufficiently large $n$ (recall that $\epsilon=
o(1)$). Therefore, we obtain that for the above mentioned $S$,
\[
\P\bigl(|\bar{E}^\star_{\cH}(S)| \geq(100/\epsilon) K \bigr) \leq\frac
{\exp(30 K)}{\exp((\epsilon/2) (100/\epsilon) K)} = \mathrm
{e}^{-20K} .
\]
By Part \hyperlink{item-lem-kernel-1}{(1)} of Lemma \ref
{lem-kernel-size-volume}, there are at most $(\frac43+o(1))\epsilon^3
n \exp(5K)$ connected sets of size $s$. Taking a union bound over the
$(\frac43+o(1))\epsilon^3 n$ values of $s$ establishes that the
statement of the lemma holds except with probability
\[
\biggl(\frac43 + o(1) \biggr)\epsilon^3 n \sum_s \mathrm{e}^{-20K(s)}\mathrm
{e}^{5K(s)} \leq
\biggl(\frac{16}9 + o(1) \biggr)(\epsilon^3 n)^{-13} = o(1) ,
\]
completing the proof of the statement on all connected subsets
$S\subset\K$.

Next, if $T$ contains $t$ edges in $\K$, then the number of
corresponding edges in $\cH$ is again stochastically dominated
by a sum of i.i.d.\ geometric variables~$\{Z_i\}$ as above. Hence, by
the same argument, the probability that there exists a set $T\subset
E(\K)$ of $\alpha\epsilon^3 n$ edges in $\K$, which expands to at
least $\beta\epsilon^2 n$ edges in $\cH$ for some $0<\alpha<\frac
12$ and $0<\beta< 1$, is at most
\[
\pmatrix{(2+o(1))\epsilon^3 n\cr\alpha\epsilon^3 n}\frac{\mathrm
{e}^{\alpha\epsilon^3 n}}{\mathrm{e}^{(\epsilon/2) \beta\epsilon
^2 n}} \leq\exp\biggl[ \biggl(2H \biggl(\frac{\alpha}2 \biggr)+\alpha-\frac{\beta
}2+o(1) \biggr) \epsilon^3 n \biggr]
\]
[using the well-known fact that $\sum_{i \leq\lambda m}{m\choose i}
\leq\exp[H(\lambda)m]$, where $H(x)$ is the entropy function $H(x) \deq- x\log
x - (1-x) \log(1-x)$].
It is now easy to verify that a choice of $\alpha=\frac1{20}$ and
$\beta=\frac23$ in the last expression yields
a term that tends to $0$ as $n\to\infty$.\vadjust{\goodbreak}

It remains to bound $|\mathcal{L}|$. This will follow from a bound on
the number of loops in $\K$.
Let $u\in\K$ be a kernel vertex, and recall that its degree $D_u$ is
distributed as an independent $ ( \operatorname{Poisson}(\Lambda) \given
\cdot\geq3 )$, where $\Lambda=(2+o(1))\epsilon$ with high probability.
The expected number of loops that $u$ obtains in a~random realization
of the degree sequence (via the configuration model) is
clearly at most $D_u^2 / D$, where $D = (4+o(1))\epsilon^3 n$ is the
total of the kernel degrees.
Therefore,
\[
\E|\mathcal{L}| \leq\bigl(\tfrac43 + o(1)\bigr)\epsilon^3 n\cdot(1/D) \E
[D_u^2] = O(1) ,
\]
and so $\E|E^\star_{\cH}(\mathcal{L})| = O(1/\epsilon)$. The
contribution of $|E^\star_{\cH}(\mathcal{L})|$ is thus easily
absorbed w.h.p.\ when increasing $\beta$ from $\frac23$ to $\frac
34$, completing the proof.
\end{pf}

\begin{lemma}\label{lem-conductance}
There exists an absolute constant $\iota> 0$ so that w.h.p. every
connected set $S \subset\cH$ with
$(200/\epsilon) \log(\epsilon^3 n)\leq d_{\cH}(S)\leq e(\cH)$
satisfies that
$|\partial_{\cH} S| / d_{\cH}(S)\geq\iota\epsilon$.
\end{lemma}
\begin{pf}
Let $S\subset\cH$ be as above, and write $S_\K= S \cap\K$. Observe
that~$S_\K$ is connected (if nonempty).
Furthermore, since $d_\cH(S) \geq(200/\epsilon)\log(\epsilon^3 n)$
whereas the longest 2-path in $\cH$
contains $(1+o(1))(1/\epsilon)\log(\epsilon^3 n)$ edges w.h.p., we
may assume that $S_\K$ is indeed nonempty.

Next, clearly $|\partial_{\cH} S| \geq|\partial_\K S_\K|$ (as each
edge in the boundary of $S_\K$ translates
into a 2-path in $\cH$ with precisely one endpoint in $S$), while
$|\bar{E}_{\cH}(S)| \leq|\bar{E}^\star_{\cH}(S_\K)|$ (any $e \in
\bar{E}_\cH(S)$ belongs to some 2-path $\mathcal{P}_e$, which is necessarily
incident to some $v\in S_\K$ as, crucially, $S_\K$ is nonempty.
Hence, the edge corresponding to $\mathcal{P}_e$ belongs to $\bar
{E}_\K(S_\K)$, and so $e \in\bar{E}^\star_\cH(S_\K)$).
Therefore, using the fact that $d_\cH(S) \leq2|\bar{E}_\cH(S)|$,
%
\begin{equation}\label{eq-decompose-cheeger}
\frac{|\partial_{\cH} S|}{d_{\cH}(S)} \geq
\frac{|\partial_\K S_\K|}{2|\bar{E}^\star_{\cH}(S_\K)|} =
\frac{|\partial_\K S_\K|}{2|\bar{E}_\K(S_\K)|}\cdot
\frac{|\bar{E}_\K(S_\K)|}{|\bar{E}_{\cH}^\star(S_\K)|} .
\end{equation}
Assume that the events stated in Lemma \ref{lem-conductance-1} hold.
Since the assumption on~$d_\cH(S_\K)$ gives that
$|\bar{E}^\star_{\cH}(S_\K)| \geq(100/ \epsilon) \log
(\epsilon^3 n)$, we deduce that
necessarily
\[
|S_\K| \geq(\epsilon/100) |\bar{E}^\star_{\cH
}(S_\K)| ,
\]
and thus (since $S_\K$ is connected)
%
\begin{equation}\label{eq-cheeger-1}|\bar{E}_\K(S_\K)| \geq|E_\K
(S_\K)| \geq(\epsilon/100) |\bar{E}^\star_{\cH}(S_\K)|-1 .
\end{equation}
Now,
\[
d_{\cH}(S) \leq e(\cH) = \bigl(2+o(1)\bigr)\epsilon^2 n ,
\]
and since $d_{\cH}(S) = 2|E_\cH(S)|+|\partial_\cH S|$
we have $|E_\cH(S)| \leq(1+o(1))\epsilon^2 n$.
In particular,
$|E(\cH) \setminus E_{\cH}(S)| \geq\frac34 \epsilon^2 n$ for
sufficiently large $n$.

At the same time, if $\mathcal{L}$ is the set of all loops in $\K$
and $T=\bar{E}_\K(\K\setminus S_\K)$,
then clearly $E^\star_{\cH}(T) \cup E^\star_{\cH}(\mathcal{L})$ is
a superset of
$E(\cH) \setminus E_{\cH}(S)$.
Therefore, Lemma~\ref{lem-conductance-1} yields that $|T| \geq\frac
1{20}\epsilon^3 n$. Since $d_\K(S_\K) \leq2e(\K) =
(4+o(1))\epsilon^3 n$, we get
\[
d_\K(\K\setminus S_\K) \geq|T| \geq \frac{\epsilon^3 n}{20} \geq
\frac{1+o(1)}{80}d_\K(S_\K) .
\]
At this point, by
Lemma \ref{lem-kernel-expander} there exists $\alpha> 0$ such
that w.h.p.\ for any such above mentioned subset $S$:
%
\begin{equation}\label{eq-cheeger-2}
|\partial_\K S_\K| \geq\alpha \bigl(d_\K(S_\K) \wedge d_\K(\K
\setminus S_\K) \bigr) \geq\frac{\alpha+o(1)}{80} d_\K(S_\K) .
\end{equation}
Plugging \eqref{eq-cheeger-1}, \eqref{eq-cheeger-2} into \eqref
{eq-decompose-cheeger}, we
conclude that the lemma holds for any sufficiently large $n$ with, say,
$\iota= \frac12\cdot10^{-4}\alpha$.
\end{pf}

We are now ready to establish the upper bound on the mixing time for
the random walk on $\cH$.

\begin{pf*}{Proof of Theorem \ref{thm-core-mixing}}
We will apply the following recent result of~\cite{FR1}, which bounds
the mixing time of a lazy chain in terms of its isoperimetric profile
(a fine-tuned version of the Lov\'{a}sz--Kannan \cite{LK} bound on the
mixing time in terms of the average conductance).
\begin{theorem}[(\cite{FR1})]\label{thm-Fountoulakis-Reed}
Let $P = (p_{x, y})$ be the transition kernel of an irreducible,
reversible and aperiodic Markov chain on $\Omega$ with stationary
distribution~$\pi$.
Let $\pi_{\min} = \min_{x\in\Omega}\pi(x)$ and for $p>\pi_{\min
}$, let
\[
\Phi(p) \deq\min\{\Phi(S)\dvtx S \mbox{ is connected and } p/2 \leq
\pi(S) \leq p\} ,
\]
and $\Phi(p)=1$ if there is no such $S$. Then for some absolute
constant $C>0$,
\[
\tmixu\leq C \sum_{j=1}^{\lceil\log\pi_{\min}^{-1}\rceil} \Phi
^{-2}(2^{-j}) .
\]
\end{theorem}

In our case, the $P$ is the transition kernel of the lazy random walk
on $\cH$.
By definition, if $S\subset\cH$ and $d_\cH(x)$ denotes the degree of
$x\in\cH$, then
\[
\pi_\cH(x) = \frac{d_\cH(x)}{2e(\cH)} ,\qquad  p_{x,y} = \frac1{2d_\cH
(x)} ,\qquad  \pi_\cH(S) = \frac{d_\cH(S)}{2e(\cH)} ,
\]
and so $\Phi(S) \geq\frac12 |\partial_\cH S| /d_\cH(S)$.
Recall that w.h.p.\ $e(\cH) = (2+o(1)) \epsilon^2
n$. Under this assumption, for any $p \geq120\frac{
\log(\epsilon^3 n)}{\epsilon^3 n}$ and connected subset $S \subset
\cH$ satisfying $\pi_{\cH}(S) \geq p/2$,
\[
d_{\cH}(S) = 2 \pi_\cH(S) e(\cH) \geq(200/ \epsilon) \log
(\epsilon^3 n) .
\]
Therefore, by Lemma \ref{lem-conductance}, w.h.p.
%
\begin{equation}\label{eq-Phi-p-bound}
\Phi(p) \geq\frac12\iota \epsilon \qquad \mbox{for all } 120 \frac{\log(\epsilon^3 n)}{\epsilon
^3 n} \leq p \leq\frac12 .
\end{equation}
Set
\[
j^* = \max \biggl\{j\dvtx 2^{-j} \geq120\frac{ \log(\epsilon^3 n)}{\epsilon^3 n} \biggr\} .
\]
It is clear that $j^* = O(\log(\epsilon^3 n))$ and \eqref
{eq-Phi-p-bound} can be translated into
%
\begin{equation}\label{eq-Phi-j-small}\Phi(2^{-j}) \geq\tfrac
12\iota\epsilon,\qquad  \mbox{for all } 1 \leq j \leq j^* .
\end{equation}
On the other hand, if $\pi_\cH(S)\leq p < 1$ then $d_\cH(S) \leq
2pe(\cH)$ while \mbox{$|\partial_\cH S|\geq1$} (as $\cH$ is connected),
and so
the inequality $\Phi(S) \geq\frac12 |\partial_\cH S| /d_\cH(S)$ gives
$\Phi(S) \geq1/(4 p e(\cH))$. Substituting $p = 2^{-j}$ with $j \leq
\lceil\log\pi_{\min}^{-1}\rceil$, we have
%
\begin{equation}\label{eq-Phi-j-large}
\Phi(2^{-j}) \geq
\frac{2^{j-2}}{e(\cH)}\geq\frac{2^j}{10\epsilon^2 n}
\end{equation}
(where the last inequality holds for large $n$).
Combining \eqref{eq-Phi-j-small} and \eqref{eq-Phi-j-large} together,
we now apply Theorem \ref{thm-Fountoulakis-Reed} to conclude that
there exists a constant $C > 0$ such that, w.h.p.,
\begin{eqnarray*}
\tmixu& \leq& C \sum_{j=1}^{\lceil\log\pi_{\min}^{-1}\rceil}
\frac{1}{\Phi^2(2^{-j})}= C \Biggl[\sum_{j=1}^{j^*} \frac{1}{\Phi
^2(2^{-j})} + \sum_{j= j^*}^{\lceil\log\pi_{\min}^{-1}\rceil}
\frac{1}{\Phi^2(2^{-j})} \Biggr]\\
&\leq& C \biggl(j^* \biggl(\frac12\iota\epsilon\biggr)^{-2} + 2 (10 \epsilon^2 n
\cdot2^{-j^*})^2 \biggr)= O(\epsilon^{-2} \log^2 (\epsilon^3 n)) ,
\end{eqnarray*}
where the last inequality follows by our choice of $j^*$.

The lower bound on the mixing time follows immediately from the fact
that, by the definition of $\tGC$, w.h.p.\ there exists a 2-path in
$\cH$ whose length is $(1-o(1))(1/\epsilon) \log
(\epsilon^3 n)$ (see \cite{DKLP1}, Corollary 1).
\end{pf*}

\subsection{Local times for the random walk on the 2-core}
In order to extend the mixing time from the 2-core $\cH$ to
the giant component, we need to prove the following proposition.

\begin{proposition}\label{prop-expected-visits}
Let $N_{v,s}$ be the local time induced by the lazy random walk $(W_t)$
on $\cH$ to the vertex $v$ up to time $s$, that is,
\mbox{$\#\{ 0 \leq t \leq s \dvtx W_t = v\}$}. Then there exists some $C\,{>}\,0$ such
that, w.h.p., for all $s\,{>}\,0$ and any \mbox{$u,v\!\in\!\cH$},
\[
\E_u[N_{v,s}] \leq C \frac{\epsilon s}{\log(\epsilon^3 n)} +
(150/\epsilon) \log(\epsilon^3 n) .
\]
\end{proposition}

In order to prove Proposition \ref{prop-expected-visits}, we wish to
show that with positive probability the random walk $W_t$\vadjust{\goodbreak}
will take an excursion in a long 2-path before returning to $v$.
Consider some $v \in\K$ (we will later extend this analysis to the
vertices in $\cH\setminus\K$, i.e., those vertices lying on 2-paths).
We point out that proving this statement is simpler in case $D_v =
O(1)$, and most of the technical
challenge lies in the possibility that $D_v$ is unbounded. In order to
treat this point, we first show that the
neighbors of vertex $v$ in the kernel are, in some sense, distant apart.

\begin{lemma}\label{lem-new-edges-distr}
For $v\in\K$ let $\mathcal{N}_v$ denote the set of neighbors of $v$
in the kernel $\K$. Then w.h.p., for every $v \in
\K$ there exists a collection of disjoint connected subsets $\{ B_w(v)
\subset\K\dvtx w \in\mathcal{N}_v\}$, such that for all $w\in
\mathcal{N}_v$,
\[
|B_w| = \lceil(\epsilon^3 n)^{1/5} \rceil \quad \mbox{and}\quad  \diam(B_w)
\leq\tfrac12\log(\epsilon^3 n) .
\]
\end{lemma}

\begin{pf}
We may again assume \eqref{eq-N-k} and furthermore, that
\[
3 \leq D_v \leq\log(\epsilon^3 n)\qquad \mbox{for all }v\in\K.
\]
Let $v\in\K$.
We construct the connected sets $B_w$ while we reveal the structure of
the kernel $\K$ via the configuration model, as follows:
Process the vertices $w \in\mathcal{N}_v$ sequentially according to
some arbitrary order.
When processing such a vertex $w$, we expose the ball (according to the
graph metric) about it, excluding $v$ and any vertices that were
already accounted for, until its size reaches $ \lceil(\epsilon^3
n)^{1/5} \rceil$ (or until no additional new vertices can be
added).\looseness=1

It is clear from the definition that the $B_w$'s are indeed disjoint
and connected,
and it remains to prove that each $B_w$ satisfies $|B_w| = \lceil
(\epsilon^3 n)^{1/5} \rceil$ and $\diam(B_w)
\leq\log(\epsilon^3 n)$.

Let $R$ denote the tree-excess of the (connected) subset $\{v\} \cup
\bigcup_{w} B_w$ once the process is concluded.
We claim that w.h.p.\ $R \leq1$. To see this, first observe that at
any point in the above process, the sum of degrees of all the vertices
that were already exposed (including $v$ and $\mathcal{N}_v$) is at most
\[
\lceil(\epsilon^3 n)^{1/5} \rceil \log^2 (\epsilon^3 n) =
(\epsilon^3n)^{1/5+o(1)} .
\]
Hence, by the definition of the configuration model (which draws a new
half-edge between $w$ and some other vertex proportional to its degree),
$R\preceq Z$ where $Z$ is a binomial variable $\Bin((\epsilon^3
n)^{1/5+o(1)} , (\epsilon^3n)^{-4/5+o(1)} )$. This gives
\[
\P(R \geq2) = (\epsilon^3 n)^{-6/5+o(1)} .
\]
In particular, since $D_w\geq3$ for any $w \in\K$, this implies that
we never fail to grow $B_w$ to size $(\epsilon^3 n)^{1/5}$, and that
the diameter of each $B_w$ is at most that of a binary tree (possibly
plus $R\leq1$), that is, for any large $n$,
\[
\diam(B_w) \leq \tfrac{1}{5}\log_2(\epsilon^3 n) +2 \leq\tfrac
12\log(\epsilon^3 n) .
\]
A simple union bound over $v\in\K$ now completes the proof.
\end{pf}

We distinguish the following subset of the edges of the kernel, whose
paths are suitably long:
\[
\mathcal{E} \deq \biggl\{e \in E(\K) \dvtx |\mathcal{P}_e| \geq \frac
{1}{20\epsilon}\log(\epsilon^3 n) \biggr\} ,
\]
where $\mathcal{P}_e$ is the 2-path in $\cH$ that corresponds to the
edge $e \in E(\K)$. Further define
$\mathcal{Q} \subset2^\K$ to be all the subsets of vertices of $\K$
whose induced subgraph contains an edge from $\mathcal{E}$:
\[
\mathcal{Q} \deq \{S \subset\K\dvtx E_\K(S) \cap\mathcal{E} \neq
\varnothing\} .
\]
For each $e \in\K$, we define the median of its $2$-path, denoted by
$\operatorname{med}(\mathcal{P}_e)$, in the obvious manner: It is the
vertex $w\in\mathcal{P}_e$ whose distance from the two endpoints is
the same, up to at most $1$ (whenever there are two choices for this
$w$, pick one arbitrarily). Now, for each $v\in\cH$ let
\[
\mathcal{E}_v \deq \{ \operatorname{med}(\mathcal{P}_e) \dvtx e \in\mathcal
{E} , v \notin\mathcal{P}_e\} .
\]
The next lemma provides a lower bound on the effective conductance
between a vertex $v$ in the 2-core and its corresponding above defined
set $\mathcal{E}_v$. See, for example, \cite{LP} for further details on
conductances/resistances.

\begin{lemma}\label{lem-conductance-G}
Let $C_{\mathrm{eff}}(v \leftrightarrow\mathcal{E}_v)$ be the
effective conductance between a~vertex $v\in\cH$ and the set
$\mathcal{E}_v$. With high probability, for any $v \in\cH$,
\[
C_{\mathrm{eff}}(v \leftrightarrow\mathcal{E}_v)/D_v \geq \epsilon
/(100\log(\epsilon^3 n)) .
\]
\end{lemma}
\begin{pf}
In order to bound the effective conductance, we need to prove that for
any $v \in\K$, there exist $D_v$ disjoint paths of length at most
$(100/\break\epsilon) \log(\epsilon^3 n)$ leading to the set $\mathcal
{E}_v$. By Lemmas \ref{lem-conductance-1} and \ref
{lem-new-edges-distr}, it suffices to prove that w.h.p.\ for any $v\in
\K$ and $w \in\mathcal{N}_v$, we have that $E(B_w) \cap\mathcal{E}
\neq\varnothing$, where~$\mathcal{N}_v$ and $B_w$ are defined as in
Lemma \ref{lem-new-edges-distr} (in this case, the path from $v$ to
some $e \in\mathcal{E}$ within $B_w$ will have length at most $\frac
12\log(\epsilon^3n)$ in $\K$, and its length will not be exceed
$(100/\epsilon)\log(\epsilon^3n)$ after being expanded in the 2-core).

Notice that if $Y$ is the geometric variable $\Geom(1 -\mu)$ then
\[
\P\biggl(Y \geq\frac{1}{10\epsilon} \log(\epsilon^3 n) \biggr) = \mu
^{(1/10\epsilon) \log(\epsilon^3 n)} \geq(\epsilon^3 n)^{-1/10+o(1)} .
\]
Therefore, by the independence of the lengths of the 2-paths
and the fact that $|B_w| = \lceil(\epsilon^3 n)^{1/5} \rceil$, we
obtain that
\[
\P\bigl(E(B_w) \cap\mathcal{E} = \varnothing\bigr) \leq \bigl(1 - (\epsilon^3
n)^{-1/10+o(1)} \bigr)^{(\epsilon^3 n)^{1/5}} \leq\mathrm{e}^{-(\epsilon^3
n)^{1/10-o(1)}} .
\]
At this point, a union bound shows that the probability that for some
$v \in\K$ there exists some $w \in\mathcal{N}_v$, such
that $E(B_w)$ does not intersect $\mathcal{E}$, is at most
\[
\bigl(\tfrac43+o(1) \bigr)\epsilon^3 n \cdot\log(\epsilon^3 n) \cdot\mathrm
{e}^{-(\epsilon^3 n)^{1/10-o(1)}} = o(1) .
\]
\upqed
\end{pf}\eject

We are ready to prove the main result of this subsection, Proposition
\ref{prop-expected-visits}, which bounds the local times
induced by the random walk on the 2-core.

\begin{pf*}{Proof of Proposition \ref{prop-expected-visits}}
For some vertex $v\in\cH$ and subset $A\subset\cH$, let
\[
\tau^+_v \deq\min\{t > 0\dvtx W_t = v\} ,\qquad
\tau_{A} \deq\min\{t\dvtx W_t \in A\} .
\]
It is well known [see, e.g., \cite{LP}, equation (2.4)] that the
effective conductance has the following form:
\[
\P_v(\tau_{A} < \tau^+_v) = \frac{C_{\mathrm{eff}}(v
\leftrightarrow A)}{D_v} .
\]
Combined with Lemma \ref{lem-conductance-G}, it follows that
\[
\P_{v}(\tau_{\mathcal{E}_v} < \tau^+_v) = \frac{C_{\mathrm
{eff}}(v \leftrightarrow\mathcal{E}_v)}{D_v}\geq\epsilon/ (100 \log
(\epsilon^3 n) ) .
\]
On the other hand, for any $v\in\cH$, by definition $w\in\mathcal
{E}_v$ is the median of some 2-path, which does not contain $v$ and has
length at least $\frac{1}{20\epsilon}\log(\epsilon^3 n)$. Hence, by
well-known properties of hitting times for the simple random walk on
the integers, there exists some absolute constant $c>0$ such
that for any $v \in\cH$ and $w \in\mathcal{E}_v$:
\[
\P_w \bigl(\tau^+_v \geq c \epsilon^{-2} \log^2(\epsilon^3 n) \bigr) \geq\P
_w \bigl(\tau_{\K} \geq c \epsilon^{-2} \log^2(\epsilon^3 n) \bigr)\geq
\tfrac23 .
\]
Altogether, we conclude that
\begin{eqnarray*}
\P_v \bigl(\tau^+_v \geq c \epsilon^{-2} \log^2(\epsilon
^3 n) \bigr) &\geq&\P_v (\tau_{\mathcal{E}_v} < \tau^+_v ) \min_{w \in
\mathcal{E}_v} \bigl\{\P_w \bigl(\tau^+_v \geq c \epsilon^{-2} \log
^2(\epsilon^3 n)\bigr) \bigr\}\\
& \geq&\epsilon/ (150 \log(\epsilon^3 n) ) .
\end{eqnarray*}
Setting $t_c = c \epsilon^{-2} \log^2 (\epsilon^3 n)$, we can
rewrite the above as
\[
\P_v(N_{v,t_c}\geq2) \leq1 - \epsilon/ (150 \log(\epsilon^3 n) ) .
\]
By the strong Markovian property (i.e., $(W_{\tau_v^+ + t})$ is a
Markov chain with the same transition kernel of $(W_t)$), we deduce
that
\[
\P(N_{v,t_c} \geq k) \leq [1 - \epsilon/ (150 \log(\epsilon^3 n) )
]^{k-1} ,
\]
and hence
\[
\E N_{v,t_c} \leq(150/\epsilon) \log(\epsilon^3 n) .
\]
The proof is completed by observing that $\E_v (N_{v,s}) \leq
\lceil s/t_c\rceil\E_v N_{v,t_c} $ and that $\E_u N_{v,s} \leq\E_v
N_{v,s}$ for any $u$.
\end{pf*}

\section{Mixing on the giant component}\label{sec:mixing-gc}

In this section, we prove Theorem~\ref{mainthm-super}, which
establishes the order of the mixing time
of the lazy random walk on the supercritical $\GC$.

\subsection{Controlling the attached Poisson Galton--Watson
trees}\label{subsec:pgw}
So far, we have established that w.h.p.\
the mixing time of the lazy random walk on the 2-core $\TC[\tGC]$ has
order $\epsilon^{-2}
\log^2(\epsilon^3 n)$. To derive the mixing time for $\tGC$ based on
that estimate,
we need to consider the delays due to the excursions the random walk
makes in the attached trees.
As we will later see, these delays will be upper bounded by a certain a
linear combination of the sizes of the trees
(with weights determined by the random walk on the 2-core). The
following lemma will play a role in estimating this expression.
%
\begin{lemma}\label{lem-average-delay}
Let $\{\mathcal{T}_i\}$ be independent $\mathrm{PGW}(\mu)$-trees.
For any two constants $C_1, C_2 > 0$ there exists some constant $C>0$
such that the following holds: If $\{a_i\}_{i=1}^m$ is a sequence of
positive reals satisfying
%
\begin{eqnarray}
\sum_{i=1}^{m} a_ i &\leq& C_1 \epsilon^{-2} \log^2(\epsilon^3 n),\label{eq-lem-avg-delay-sum}\\
\max_{1\leq i \leq m} a_i &\leq& C_2 \epsilon^{-1} \log(\epsilon^3n) ,\label{eq-lem-avg-delay-max}
\end{eqnarray}
then
\[
\P \Biggl(\sum_{i=1}^{m} a_i |\mathcal{T}_i| \geq C \epsilon^{-3} \log^2
(\epsilon^3 n) \Biggr) \leq(\epsilon^3 n)^{-2} .
\]
\end{lemma}

\begin{pf}
It is well known (see, e.g., \cite{Pitman}) that the size of a
Poisson$(\gamma)$--Galton--Watson tree $\mathcal{T}$ follows a
Borel($\gamma$) distribution, namely,
%
\begin{equation}\label{eq-PGW-size}
\P(|\mathcal{T}| = k)=\frac{k^{k-1}}{\gamma k!}(\gamma\mathrm
{e}^{-\gamma})^{k} .
\end{equation}
The following is a well-known (and easy) estimate on the size of a
PGW-tree; we include its proof for completeness.
\begin{claim}\label{claim-variance-PGW}
Let $0<\gamma< 1$, and let $\mathcal{T}$ be a $\mathrm{PGW}(\gamma
)$-tree. Then
\[
\E|\mathcal{T}| = \frac{1}{1-\gamma} ,\qquad  \var( |\mathcal{T}| ) =
\frac{\gamma}{(1-\gamma)^3} .
\]
\end{claim}

\begin{pf}
For $k = 0, 1, \ldots,$ let $L_k$ be the number of vertices in the
$k$th level of the tree $\mathcal{T}$. Clearly,
$\E L_k = \gamma^k$, and so
$\E|\mathcal{T}| = \E\sum_kL_k = 
\frac{1}{1-\gamma}$.

By the total-variance formula,
%
\begin{eqnarray*}
\var(L_i) &=& \var(\E(L_i \given L_{i-1} ) ) + \E(\var(L_i \given
L_{i-1} ) )\\
&=&\gamma^2 \var(L_{i-1}) + \gamma\E L_{i-1} = \gamma^2 \var
(L_{i-1})+ \gamma^i .
\end{eqnarray*}
%
By induction,
%
\begin{equation}\label{eq-var-L-i}
\var(L_i) = \sum_{k=i}^{2i-1} \gamma^k = \gamma^i \frac{1-\gamma
^i}{1-\gamma} .\vadjust{\goodbreak}
\end{equation}
We next turn to the covariance of $L_i,L_j$ for $i \leq j$:
\begin{eqnarray*}
\Cov(L_i, L_j)& = &\E[L_i L_j] - \E L_i \E L_j 
= \gamma^{j-i} \E L_i^2 - \gamma^{i+j} \\ 
& = &\gamma^{j-i} \var(L_i) = \gamma^j \frac{1 - \gamma^i}{1 -
\gamma} .
\end{eqnarray*}
Summing over the variances and covariances of the $L_i$'s, we deduce that
%
\[
\var(|\mathcal{T}|) 
= 2\sum_{i= 0}^\infty\sum_{j= i}^\infty\gamma^{j}\frac{1 - \gamma
^i}{1 - \gamma} - \sum_{i=0}^\infty\gamma^i \frac{1-\gamma
^i}{1-\gamma} 
=\frac{\gamma}{(1-\gamma)^3} .
\]
\upqed
\end{pf}
We need the next lemma to bound the tail probability for $\sum a_i
|\mathcal{T}_i|$.
\begin{lemma}[(\cite{KimB}, Corollary 4.2)]\label{lem-Generalized-Chernoff-bound}
Let $X_1,\ldots, X_m$ be independent r.v.'s with $\E[X_i] = \mu_i$.
Suppose there are $b_i$, $d_i$ and
$\xi_0$ such that $\var(X_i) \leq b_i$, and
\[
\bigl|\E\bigl[(X_i - \mu_i)^3 \mathrm{e}^{\xi(X_i - \mu_i)} \bigr] \bigr| \leq d_i
\qquad \mbox{for all } 0\leq|\xi| \leq\xi_0 .
\]
If $\delta\xi_0 \sum_{i=1}^m d_i \leq\sum_{i=1}^{m} b_i$ for some
$0 < \delta\leq1$, then for all $\Delta>0$,
\[
\P \Biggl( \Biggl|\sum_{i=1}^m X_i - \sum_{i=1}^m \mu_i \Biggr| \geq\Delta\Biggr) \leq
\exp\biggl(-\frac{1}{3}\min\biggl\{\delta\xi_0 \Delta, \frac{\Delta^2}{\sum
_{i=1}^m b_i} \biggr\} \biggr) .
\]
\end{lemma}

Let $T_i = |\mathcal{T}_i|$ and $X_i = a_i T_i$ for $i \in[m]$. Claim
\ref{claim-variance-PGW} gives that
\[
\mu_i = \E X_i = a_i/(1-\mu) .
\]
Now set
\[
\xi_0 = \epsilon^3/(10 C_2 \log(\epsilon^3 n)) .
\]
For any $|\xi| \leq\xi_0$, we have $a_i |\xi| \leq\epsilon^2/10$
by the assumption \eqref{eq-lem-avg-delay-max}, and so
\begin{eqnarray}\label{eq-e-(xi-mu)^3}
\bigl|\E\bigl[(X_i - \mu_i)^3 \mathrm{e}^{\xi(X_i - \mu_i)} \bigr] \bigr| &=& a_i^3
\biggl|\E\biggl[ \biggl(T_i - \frac{1}{1-\mu} \biggr)^3 \mathrm{e}^{\xi a_i (T_i -
{1}/{(1-\mu)})} \biggr]\biggr|\nonumber\\
& \leq& a_i^3 \E\bigl[(1-\mu)^{-3} \one_{\{T_i < (1-\mu)^{-1}\}} \bigr] \nonumber\\[-8pt]\\[-8pt]
&&{}+a_i^3
\E\bigl[T_i ^3 \mathrm{e}^{\xi a_i T_i}\one_{\{T_i\geq(1-\mu)^{-1}\}}\bigr]\nonumber \\
& \leq& a_i^3 (1-\mu)^{-3} + a_i^3 \E [T_i^3 \exp(\epsilon^2
T_i/10) ] .\nonumber
\end{eqnarray}
Recalling the law of $T_i$ given by \eqref{eq-PGW-size}, we obtain that
\[
\E \bigl(T_i^3 \exp(\epsilon^2 T_i/10) \bigr)= \sum_{k = 1}^\infty\frac
{k^{k-1}}{\mu k!} (\mu\mathrm{e}^{-\mu})^k k^3 \mathrm{e}^{\epsilon^2
k/10} .
\]
Using Stirling's formula, we obtain that for some absolute constant $c>1$,
\begin{eqnarray}\label{eq-apply-stirling}
\E\bigl(T_i^3 \exp(\epsilon^2 T_i/10) \bigr)
&\leq& c\sum_{k = 1}^\infty\frac{ k^{k-1} (\mu\mathrm{e}^{-\mu
})^k}{\mu(k/\mathrm{e})^k \sqrt{k}}k^3 \mathrm{e}^{\epsilon^2
k/10}\nonumber\\[-8pt]\\[-8pt]
&=&\frac{c}{\mu} \sum_{k = 1}^\infty k^{3/2} (\mu\mathrm{e}^{1-\mu
})^k \mathrm{e}^{\epsilon^2 k/10} .\nonumber
\end{eqnarray}
Recalling that $\mu = 1-\epsilon+ \frac23 \epsilon^2 + O(\epsilon
^3)$ and
using the fact that $1 - x \leq\mathrm{e}^{-x-x^2/2}$ for $x \geq
0$, we get that for sufficiently large $n$ (and hence small enough
$\epsilon$),
%
\begin{equation}\label{eq-mu-e-1-mu}
\mu\mathrm{e}^{1 - \mu} = \bigl(1-(1-\mu) \bigr)\mathrm{e}^{1-\mu} \leq\exp\bigl(-\tfrac12\epsilon^2 +
O(\epsilon^3) \bigr) \leq\mathrm{e}^{-\epsilon^2/3} .
\end{equation}
Plugging the above estimate into \eqref{eq-apply-stirling}, we obtain
that for large $n$,
\begin{eqnarray*}
\E[T_i^3 \exp(\epsilon^2 T_i/10) ]&\leq&2c \sum_{k =1}^\infty
k^{3/2} \mathrm{e}^{-\epsilon^2 k/6} \leq
4c\int_0^\infty x^{3/2} \mathrm{e}^{-\epsilon^2 x/6} \,dx\\
& \leq&400c\epsilon^{-5}\int_0^\infty x^{3/2}\mathrm{e}^{-x} \,dx =
300\sqrt{\pi} c \epsilon^{-5} .
\end{eqnarray*}
Going back to \eqref{eq-e-(xi-mu)^3}, we get that for some absolute
$c'>1$ and any large $n$,
\[
\bigl|\E\bigl[(X_i - \mu_i)^3 \mathrm{e}^{\xi(X_i - \mu_i)} \bigr] \bigr| \leq a_i^3
(2\epsilon^{-3} + c'\epsilon^{-5} ) \leq a_i \cdot2c' C_2^2 \epsilon
^{-7} \log^2 (\epsilon^3 n)\deq d_i ,
\]
where the second inequality used \eqref{eq-lem-avg-delay-max}.

By Claim \ref{claim-variance-PGW}, it follows that for large enough $n$,
\[
\var(X_i) = a_i^2 \var(T_i) = a_i^2 \frac{\mu}{(1 - \mu)^3} \leq2
a_i^2 \epsilon^{-3} \leq a_i \cdot2 C_2 \epsilon^{-4} \log(\epsilon
^3 n) \deq b_i .
\]
Since $\sum_i d_i = (c' C_2 \epsilon^{-3} \log(\epsilon^3 n) )\sum
_i b_i$, by setting $\delta=1$ (and recalling our choice of $\xi_0$)
we get
\[
\delta\xi_0 \sum_{i=1}^m d_i = \frac{\delta c'}{10}\sum_i b_i \leq
\sum_{i=1}^{m} b_i .
\]
We have thus established the conditions for Lemma \ref
{lem-Generalized-Chernoff-bound}, and it remains to select~$\Delta$.
For a choice of $\Delta= (60 C_2 \vee\sqrt{12 C_1 C_2}) \epsilon
^{-3} \log^2 (\epsilon^3 n)$, by definition of $\xi_0$ and the
$b_i$'s we have
\begin{eqnarray*}
\xi_0 \Delta&\geq&6 \log(\epsilon^3 n) ,\\
\Delta^2/\sum_i b_i &\geq&6 C_1 \epsilon^{-2} \log
^3(\epsilon^3 n) / \sum_i a_i \geq6\log(\epsilon^3 n) ,
\end{eqnarray*}
where the last inequality relied on \eqref{eq-lem-avg-delay-sum}.
Hence, an application of Lemma~\ref{lem-Generalized-Chernoff-bound}
gives that
for large enough $n$,
\[
\P \biggl(\sum_i a_i T_i - \sum_i \mu_i \geq\Delta\biggr) \leq
(\epsilon^3 n)^{-2} .
\]
Finally, by \eqref{eq-lem-avg-delay-sum} and using the fact that
$1-\mu\geq\epsilon/2$ for any large $n$,
we have $\sum_i \mu_i = (1-\mu)^{-1}\sum_i a_i \leq2 C_1 \epsilon
^{-3} \log^2(\epsilon^3 n)$.
The proof of Lemma \ref{lem-average-delay} is thus concluded by
choosing $C = 2C_1 + (60 C_2 \vee \sqrt{12 C_1 C_2})$.
\end{pf}

To bound the time it takes the random walk to exit from an attached
PGW-tree (and enter the 2-core), we will need to control the diameter
and volume of such a tree. The following simple lemma of \cite{DKLP2}
gives an estimate on the diameter of a PGW-tree.
\begin{lemma}[(\cite{DKLP2}, Lemma 3.2)]\label{lem-PGW-diameter}
Let $\mathcal{T}$ be a $\mathrm{PGW}(\mu)$-tree and $L_k$ be its
$k$th level of vertices. Then
$ \P(L_k \neq\varnothing) \asymp\epsilon\mathrm{e}^{-k(\epsilon+
O(\epsilon^2))}$ for any $k \geq1/\epsilon$.
\end{lemma}

The next lemma gives a bound on the volume of a PGW-tree.
\begin{lemma}\label{lem-PGW-volume}
Let $\mathcal{T}$ be a $\mathrm{PGW}(\mu)$-tree. Then
\[
\P\bigl(|\mathcal{T}| \geq6\epsilon^{-2} \log(\epsilon^3 n)\bigr) = o(
\epsilon(\epsilon^3 n)^{-2}) .
\]
\end{lemma}

\begin{pf}
Recalling \eqref{eq-PGW-size} and
applying Stirling's formula, we obtain that for any $s > 0$,
%
\begin{equation}\label{eq-PGW-tail-prob}
\P(|\mathcal{T}| \geq s) = \sum_{k \geq s} \frac{k^{k-1}}{\mu k!}
(\mu\mathrm{e}^{-\mu})^k
\asymp\sum_{k \geq s}
\frac{(\mu\mathrm{e}^{1 - \mu})^k}{k^{3/2}} .
\end{equation}
Write $r =\log(\epsilon^3 n) $. By estimate \eqref{eq-mu-e-1-mu}, we
now get that for large enough $n$,
\[
\sum_{k \geq 6 \epsilon^{-2} r} k^{-3/2}(\mu\mathrm{e}^{1 -
\mu})^k \leq\sum_{k \geq 6 \epsilon^{-2} r}
k^{-3/2}\mathrm{e}^{-\epsilon^2 k/3}
= O\bigl( \mathrm{e}^{-2r } \epsilon/\sqrt{r}\bigr) ,
\]
and combined with \eqref{eq-PGW-tail-prob} this concludes the proof.
\end{pf}

Finally, for the lower bound, we will need to show that w.h.p.\ one of the
attached PGW-trees in $\tGC$ is suitably large, as we next describe.
For a rooted tree $\mathcal{T}$, let $L_k$ be its $k$th level of
vertices and $\mathcal{T}_v$ be its entire subtree rooted at $v$.
Define the event
\[
A_{r, s}(\mathcal{T}) \deq\{\mbox{$\exists v \in L_r$ such that
$|\mathcal{T}_v| \geq s$}\} .
\]
The next lemma gives a bound on the probability of this event when
$\mathcal{T}$ is a~PGW($\mu$)-tree.
\begin{lemma}\label{lem-bush-diam-volume}
Let $\mathcal{T}$ be a $\mathrm{PGW}(\mu)$-tree and take
$r = \lceil\frac{1}{8} \epsilon^{-1} \log(\epsilon^3 n) \rceil$
and $s= \frac{1}{8} \epsilon^{-2} \log(\epsilon^3 n)$.
Then for any sufficiently large $n$,
\[
\P(A_{r, s}(\mathcal{T})) \geq\epsilon(\epsilon^3 n)^{-2/3} .
\]
\end{lemma}

\begin{pf}
We first give a lower bound on the probability that $|\mathcal{T}|
\geq s$. By~\eqref{eq-PGW-tail-prob}, we have $\P(|\mathcal{T}| \geq
s) \geq c\sum_{k \geq s}k^{-3/2} (\mu\mathrm{e}^{1-\mu})^k$ for
some absolute $c>0$.
Recalling that $\mu = 1-\epsilon+ \frac23 \epsilon^2 + O(\epsilon
^3)$, we have that for $n$ large enough,
\[
\mu\mathrm{e}^{1 - \mu} \geq\mathrm{e}^{-(\epsilon+ \epsilon^2)}
\mathrm{e}^{\epsilon- \epsilon^2} \geq\mathrm{e}^{-2\epsilon^2} .
\]
Therefore, for $s = \frac{1}{8} \epsilon^{-2} \log(\epsilon^3 n)$
this gives that
\[
\P(|\mathcal{T}| \geq s) \geq c \sum_{s \leq k \leq2s}k^{-3/2}
\mathrm{e}^{-2 \epsilon^2 k} \geq c s (2s)^{-3/2} \mathrm{e}^{-4
\epsilon^2 s} \geq
\epsilon(\epsilon^3 n)^{-1/2+o(1)} .
\]
Combining this with the fact that $\{\mathcal{T}_v \dvtx v \in L_r\}$ are
i.i.d.\ PGW($\mu$)-trees given~$L_r$, we get
\[
\P(A_{r, s}(\mathcal{T}) \given L_r )
=1 - \bigl(1 - \P(|\mathcal{T}| \geq s)\bigr)^{|L_r|} \geq1- \bigl(1 - \epsilon
(\epsilon^3 n)^{-1/2+o(1)}\bigr)^{|L_r|} .
\]
Taking expectation over $L_r$, we conclude that
\begin{eqnarray}\label{eq-bush-diam-temp}
\P(A_{r, s}(\mathcal{T}) ) & \geq&1 - \E \bigl(\bigl(1 - \epsilon(\epsilon^3
n)^{-1/2+o(1)}\bigr)^{|L_r|} \bigr)\nonumber\\[-8pt]\\[-8pt]
& \geq&\epsilon(\epsilon^3 n)^{-1/2+o(1)} \E|L_r| - \epsilon^2
(\epsilon^3 n)^{-1+o(1)} \E|L_r|^2 .\nonumber
\end{eqnarray}
For $r = \lceil\frac{1}{8} \epsilon^{-1} \log(\epsilon^3 n) \rceil
$, we have
\[
\E(|L_r|) = \mu^r \geq\mathrm{e}^{-(\epsilon+ O(\epsilon^2))r}
\geq(\epsilon^3 n)^{-1/8+o(1)} ,
\]
and by \eqref{eq-var-L-i},
\[
\var|L_r| =\mu^r \frac{1- \mu^r}{1 - \mu} \leq\mathrm
{e}^{-\epsilon r} 2 \epsilon^{-1} \leq 2 \epsilon^{-1}(\epsilon^3
n)^{-1/8} .
\]
Plugging these estimates into \eqref{eq-bush-diam-temp}, we obtain that
\[
\P(A_{r, s}(\mathcal{T}) ) \geq\epsilon(\epsilon^3 n)^{-5/8+o(1)}
\geq\epsilon(\epsilon^3 n)^{-2/3} ,
\]
where the last inequality holds for large enough $n$, as required.
\end{pf}

\subsection{\texorpdfstring{Proof of Theorem \protect\ref{mainthm-super}: Upper bound on the mixing
time}{Proof of Theorem 1: Upper bound on the mixing time}}
By Theorem~\ref{thm-struct-gen}, it suffices to consider $\tGC$
instead of $\GC$.
As in the previous section, we abbreviate $\TC[\tGC]$ by $\cH$.

For each vertex $v$ in the 2-core $\cH$, let $\mathcal{T}_v$ be the
PGW-tree attached
to $v$ in $\tGC$. Let $(S_t)$ be the lazy random walk on $\tGC$,
define $\xi_0
= 0$ and for $j\geq0$,
\[
\xi_{j+1} = \cases{
\xi_{j}+1, &\quad  if $S_{\xi_{j}+1} = S_{\xi_{j}} $,\cr
\min \{ t > \xi_{j} \dvtx S_t \in\cH, S_t \neq S_{\xi_{j}} \},
&\quad otherwise.
}
\]
Defining $W_j \deq S_{\xi_{j}}$, we observe that $(W_j)$ is a lazy random
walk on $\cH$. Furthermore, started from any $w \in\cH$, there are
two options:
\begin{longlist}[(ii)]
\item[(i)] Do a step in the 2-core (either stay in $w$ via the lazy rule,
which has probability $\frac12$, or jump
to one of the neighbor of $w$ in $\cH$, an event that has probability
$d_\cH(w) /2d_\tGC(w)$).
\item[(ii)] Enter the PGW-tree attached to $w$ (this happens with probability
$d_{\mathcal{T}_w}(w)/ 2d_\tGC(w)$).
\end{longlist}
It is the latter case that incurs a delay for the random walk on $\tGC$.
Since the expected return time to $w$ once entering the tree $\mathcal
{T}_w$ is $2(|\mathcal{T}_w|-1)/d_{\mathcal{T}_w}(w)$,
and as the number of excursions to the tree follows a geometric
distribution with success probability $1- d_{\mathcal{T}_w}(w)/2d_\tGC(w)$,
we infer that
\[
\E_w \xi_1 = 1+ \frac{2(|\mathcal{T}_w|-1)}{d_{\mathcal{T}_w}(w)}
\cdot\frac{2d_\tGC(w)}{2d_\tGC(w)- d_{\mathcal{T}_w}(w)}
\leq4|\mathcal{T}_w| .
\]

For some constant $C_1>0$ to be specified later, let
%
\begin{equation}\label{eq-ell-def}
\ell = C_1 \epsilon^{-2} \log^2(\epsilon^3 n) \quad  \mbox{and}\quad
a_{v,w}(\ell) = \sum_{j=0}^{\ell-1}\P_v(W_j = w) .
\end{equation}
It follows that
\begin{eqnarray}\label{eq-E-tau-ell}
\E_v (\xi_\ell)& = & \sum_{j=0}^{\ell-1}\sum_{w\in\cH} \P
_v(S_{\xi_j} = w) \E_w \xi_1 \nonumber\\[-8pt]\\[-8pt]
&=& \sum_{w\in\cH} \sum_{j=0}^{\ell-1}\P_v(W_j = w) \E_w \xi_1
\leq4\sum_{w\in\cH} a_{v, w}(\ell)|\mathcal{T}_w| .\nonumber
\end{eqnarray}
We now wish to bound the last expression via Lemma \ref{lem-average-delay}.
Let $v \in\K$. Note that, by definition,
\[
\sum_{w\in\cH}a_{v, w}(\ell) =\ell=C_1 \epsilon^{-2} \log
^2(\epsilon^3 n) .
\]
Moreover, by Proposition \ref{prop-expected-visits}, there
exists some constant $C_2>0$ (which depends on $C_1$) such that w.h.p.\
\[
\max_{w\in\cH}
a_{v, w}(\ell) \leq C_2 \epsilon^{-1} \log(\epsilon^3 n) .
\]
Hence, Lemma \ref{lem-average-delay} (applied on the sequence $\{
a_{v,w}(\ell) \dvtx w \in\cH\}$) gives that there exists some constant
$C>0$ (depending only on $C_1,C_2$) such that
\[
\sum_{w \in\cH} a_{v, w}(\ell) |\mathcal{T}_v| \leq C \epsilon
^{-3} \log^2(\epsilon^3 n)\qquad
\mbox{except with probability $(\epsilon^{3}n)^{-2}$} .
\]
Since $|\K| = (\frac43 + o(1)) \epsilon^{3} n$ w.h.p., taking a
union bound over the vertices of the kernel
while recalling \eqref{eq-E-tau-ell} implies that w.h.p.,
%
\begin{equation}\label{eq-bound-tau-j}
\E_v (\xi_\ell) \leq C \epsilon^{-3} \log^2(\epsilon^3 n) \qquad \mbox{for all $v\in\K$} .
\end{equation}

We next wish to bound the hitting time to the kernel $\K$, defined next:
\[
\tau_\K= \min\{t \dvtx S_t \in\K\} .
\]
Define $\tau_{x}$ and $\tau_{S}$ analogously as the hitting times of
$S_t$ to the vertex $x$ and the subset $S$, respectively.
Recall that from any $v\in\tGC$, after time $\xi_1$ we will have hit
a vertex in the 2-core, hence
for any $v\in\tGC$ we have
%
\begin{equation}\label{eq-tau-K-1}
\E_v \tau_\K\leq\E_v \tau_{\cH} + \max_{w \in\cH} \E_w \tau_\K.
\end{equation}
To bound the first summand, since
\[
\max_{v\in\tGC}\E_v \tau_{\cH} = \max_{w\in\cH}\max_{v\in
\mathcal{T}_w} \E_v \tau_w ,
\]
it clearly suffices to bound $\E_v \tau_w$ for all $w\in\cH$ and
$v\in
\mathcal{T}_w$. To this end,
let $w \in\cH$, and let $\tilde{S}_t$ be the lazy random walk on
$\mathcal{T}_w$.
As usual, define $\tilde{\tau}_v = \min\{t\dvtx \tilde{S}_t = v\}$.
Clearly, for all
$v\in\mathcal{T}_w$ we have $\E_v \tau_w = \E_v\tilde{\tau}_w$.
We bound~$\E_v \tilde{\tau}_w$ by $\E_v \tilde{\tau}_w + \E
_w\tilde{\tau}_v$, that is, the commute time between $v$ and $w$. Denote by
$\mathrm{R}_{\mathrm{eff}}(v, w)$ the effective resistance between
$v$ and $w$ when each edge has unit resistance. The commute time
identity of \cite{CRRST} (see also \cite{Tetali}) yields that
%
\begin{equation}\label{eq-Ev-tau-w}
\E_v \tilde{\tau}_w + \E_w\tilde{\tau}_v \leq4 |\mathcal{T}_w|
R_{\mathrm{eff}(v\leftrightarrow w)} \leq4 |\mathcal{T}_w| \diam
(\mathcal{T}_w) .
\end{equation}
Now, Lemmas \ref{lem-PGW-diameter} and \ref{lem-PGW-volume} give that for
any $w \in\cH$, with probability at least $1 - O(\epsilon
(\epsilon^3 n)^{-2})$,
%
\begin{equation}\label{eq-PGW-volume-diameter}
|\mathcal{T}_w| \leq6\epsilon^{-2} \log(\epsilon^3 n) \quad \mbox{and}\quad
\diam(\mathcal{T}_w) \leq2\epsilon^{-1} \log(\epsilon^3 n) .
\end{equation}
Since w.h.p.\ $|\cH| = (2+o(1))\epsilon^2 n$, we can sum the above
over the vertices of $\cH$ and conclude that w.h.p.,
\eqref{eq-PGW-volume-diameter} holds simultaneously for all $w \in\cH$.
Plugging this in \eqref{eq-Ev-tau-w}, we deduce that
\[
\E_v \tilde{\tau}_w + \E_w\tilde{\tau}_v \leq48\epsilon
^{-3}\log^2(\epsilon^3n) ,
\]
and altogether, as the above holds for every $w \in\cH$,
%
\begin{equation}\label{eq-bounnd-tau-1} \max_{v\in\tGC}\E_v \tau
_{\cH} \leq48\epsilon^{-3}\log^2(\epsilon^3n) .
\end{equation}

For the second summand in \eqref{eq-tau-K-1}, consider $e\in\K$ and
let $\mathcal{P}_e$ be the
2-path corresponding to $e$ in the 2-core $\cH$. Recall that w.h.p.\
the longest such 2-path in the 2-core
has length $(1+o(1))\epsilon^{-1} \log(\epsilon^3 n)$.
Since from each point $v\in\mathcal{P}_e$, we have probability at
least $2/|\mathcal{P}_e|$ to hit one
of the endpoints of the 2-path (belonging to $\K$) before returning to
$v$, it follows that w.h.p.,
for every $e\in\K$ and $v\in\mathcal{P}_e$ we have
%
\begin{equation}\label{eq-max-visit}
\max_{w\in\mathcal{P}_e} \E_w \#\{t \leq\tau_\K\dvtx
W_t = v \} \leq\biggl(\frac12+o(1)\biggr)\epsilon^{-1}\log(\epsilon^3 n).\vadjust{\goodbreak}
\end{equation}
We now wish to apply Lemma \ref{lem-average-delay} to the sequence
$a_{v} = \max_{w\in\mathcal{P}_e}\E_w \#\{t \leq\tau_\K\dvtx W_t = v
\}$. Since this sequence satisfies
\[
\max_{v\in\mathcal{P}_e} a_v \leq\biggl(\frac12+o(1)\biggr)\epsilon^{-1}\log(\epsilon^3 n) ,\qquad
\sum_{v\in\mathcal{P}_e} a_v  \leq \biggl(\frac12+o(1)\biggr)\epsilon^{-2}\log^2(\epsilon^3 n) ,
\]
we deduce that there exists some absolute constant $C'>0$ such that,
except with probability $O((\epsilon^3 n)^{-2})$,
every $w\in\mathcal{P}_e$ satisfies
%
\begin{equation}\label{eq-bound-tau-K}
\E_w \tau_\K\leq C' \epsilon^{-3}\log^2{\epsilon^3 n} .
\end{equation}
Recalling that $e(\K) = (2+o(1))\epsilon^3 n$ w.h.p., we deduce that
w.h.p.\ this statement holds simultaneously for all $w\in\cH$.
Plugging \eqref{eq-bounnd-tau-1} and \eqref{eq-bound-tau-K}
into~\eqref{eq-tau-K-1} we conclude that w.h.p.\
\[
\E_v \tau_\K\leq(C' + 48) \epsilon^{-3}\log^2{\epsilon^3 n}\qquad \mbox{for all $v\in\tGC$} .
\]

Finally, we will now translate these hitting time bounds into an upper
bound on the approximate forget time for $S_t$.
Let $\pi_{\cH}$ denote the stationary measure on the walk restricted
to $\cH$:
\[
\pi_{\cH} (w) = d_{\cH}(w)/2e(\cH) \qquad \mbox{for $w\in\cH$} .
\]
Theorem \ref{thm-core-mixing} enables us to choose some absolute
constant $C_1>0$ so that $\ell$, defined in \eqref{eq-ell-def}
as $C_1 \epsilon^{-2}\log^2(\epsilon^3 n)$, would w.h.p.\ satisfy
%
\begin{equation}\label{eq-mixing-C}\max_{w\in
\cH} \Biggl\|\frac{1}{\ell}\sum_{j=1}^\ell\P_w(W_j\in\cdot) -
\pi_{\cH} \Biggr\|_{\mathrm{TV}} \leq
\frac{1}{4} .
\end{equation}
Define $\bar{\xi}_0 = \tau_\K$ and for
$j\geq0$, define $\bar{\xi}_{j+1}$ as we did for $\xi_j$'s, that is,
\[
\bar{\xi}_{j+1} = \cases{
\bar{\xi}_{j}+1, &\quad  if $S_{\bar{\xi}_{j}+1} = S_{\bar{\xi
}_{j}} $,\cr
\min \{ t > \bar{\xi}_{j} \dvtx S_t \in\cH, S_t \neq S_{\bar{\xi
}_{j}} \}, &\quad otherwise.
}
\]
Let $\Gamma$ be the stopping rule that selects $j \in\{0,\ldots,\ell
-1\}$ uniformly and then stops at $\bar{\xi}_j$.
By \eqref{eq-mixing-C}, w.h.p.
\[
\max_{v\in\tGC} \| \P_v(S_\Gamma\in\cdot) - \pi_{\cH} \|
_{\mathrm{TV}} \leq\frac{1}{4} .
\]
Going back to the definition of the approximate forget time in \eqref
{eq-def-approx-forget}, taking $\varphi= \pi_{\cH}$ with the
stopping rule $\Gamma$ yields $\mathcal{F}_{1/4} \leq\max_{v \in
\tGC} \E\Gamma\leq\max_{v \in\tGC} \bar{\xi}_\ell$.

Furthermore, combining \eqref{eq-bound-tau-j} and \eqref
{eq-bound-tau-K}, we get that w.h.p.\ for any $v\in\tGC$:
\[
\E_v \bar{\xi}_\ell\leq(C + C' + 48) \epsilon^{-3}\log
^2(\epsilon^3 n) .
\]
Altogether, we can conclude that the approximate forget time for $S_t$
w.h.p.\ satisfies that
\[
\mathcal{F}_{1/4} \leq\max_{v\in\tGC}\E_v \bar{\xi}_\ell\leq
(C + C'+48) \epsilon^{-3}\log^2(\epsilon^3 n) .\vadjust{\goodbreak}
\]
This translates into the required upper bound on $\tmix$ via an
application of Theorems \ref{thm-Cesaro-mixing-time} and \ref{thm-forget-mix}.
\qed

\subsection{\texorpdfstring{Proof of Theorem \protect\ref{mainthm-super}: Lower bound on the mixing
time}{Proof of Theorem 1: Lower bound on the mixing time}}
As before, by Theorem \ref{thm-struct-gen} it suffices to prove the
analogous statement for $\tGC$.

Let $r,s$ be as in Lemma \ref{lem-bush-diam-volume}, that is,
\[
r = \bigl\lceil\tfrac18\epsilon^{-1} \log(\epsilon^3 n)\bigr\rceil \quad
\mbox{and}\quad  s= \tfrac18\epsilon^{-2} \log(\epsilon^3 n) .
\]
Let $\mathcal{T}_v$ for $v \in\cH$ be the PGW($\mu$)-tree that is
attached to the vertex $v$. Lemma \ref{lem-bush-diam-volume} gives
that when
$n$ is sufficiently large, every $v \in\cH$ satisfies
\[
\P(A_{r, s} (\mathcal{T}_v)) \geq\epsilon(\epsilon^3 n)^{-2/3} .
\]
Since $|\cH| = (2 + o(1)) \epsilon^2 n$ w.h.p.\ (recall Theorem \ref
{thm-struct-gen}), and
since $\{\mathcal{T}_v \dvtx v\in\cH\}$ are i.i.d.\ given $\cH$, we can
conclude that w.h.p.\ there exists some
$\rho\in\cH$ such that $A_{r, s}(\mathcal{T}_{\rho})$ holds. Let
$\rho\in\cH$ therefore be such a vertex.

Let $(S_t)$ be a lazy random walk on $\tGC$ and $\pi$ be its
stationary distribution. As usual, let $\tau_v \deq\min\{t\dvtx S_t = v\}
$. We wish to prove that
%
\begin{equation}\label{eq-bound-tau-rho}
\max_{w \in\mathcal{T}_{\rho}} \P_w \biggl(\tau_{\rho} \geq\frac23
rs \biggr) \geq\frac13 .
\end{equation}

For $w \in\mathcal{T}_{\rho}$, let $\mathcal{T}_w$ be the entire
subtree rooted at $w$. Further let $L_r$ be the vertices of the $r$th
level of $\mathcal{T}_{\rho}$. By our assumption on $\mathcal
{T}_{\rho}$, there is some $\xi\in L_r$ such that $|\mathcal{T}_{\xi
}| \geq s$.

We will derive a lower bound on $\E_{\xi} \tau_{\rho}$ from the
following well-known connection between hitting-times of random walks
and flows on electrical networks (see \cite{Tetali} and also \cite{LP}, Proposition
2.19).
\begin{lemma}[(\cite{Tetali})]\label{lem-network-hitting-time}
Given a graph $G=(V,E)$ with a vertex $z$ and a subset of vertices $Z$
not containing $z$, let $v(\cdot)$ be the voltage when a unit current
flows from
$z$ to $Z$ and the voltage is $0$ on $Z$. Then $\E_z\tau_Z =\sum_{x
\in V} d(x) v(x)$.
\end{lemma}

In our setting, we consider the graph $\tGC$. Clearly, the effective
resistance between
$\rho$ and $\xi$ satisfies $R_{\mathrm{eff}}(\rho\leftrightarrow
\xi) = r$.
If a unit current flows from $\xi$ to $\rho$
and $v(\rho)=0$, it follows from Ohm's law that $v(\xi)=r$. Notice
that for any $w \in
\mathcal{T}_{\xi}$, the flow between $w$ and $\xi$ is $0$.
Altogether, we deduce that
\[
v(w) = r \qquad \mbox{for all } w\in\mathcal{T}_{\xi} .
\]
Therefore, Lemma \ref{lem-network-hitting-time} implies that
\[
\E_{\xi} \tau_{\rho} \geq r |\mathcal{T}_{\xi}| \geq r s .
\]
Clearly, if $w^\star\in\mathcal{T}_\rho$ attains $\max\{\E_w \tau
_{\rho} \dvtx w\in\mathcal{T}_{\rho}\}$ then clearly
\[
\E_{w^\star}\tau_{\rho}\leq\tfrac23 rs + \P_{w^\star} \bigl(\tau
_{\rho} \geq\tfrac23 rs \bigr)\E_{w^\star}\tau_{\rho} .
\]
On the other hand,
\[
\E_{w^\star} \tau_{\rho} \geq\E_{\xi}\tau_{\rho} \geq r s ,
\]
hence we obtain \eqref{eq-bound-tau-rho}.

Recall that w.h.p.\ $|\tGC|
= (2+o(1))\epsilon n$. Together with Lemma \ref{lem-PGW-volume}, we
deduce that w.h.p.\ every $v \in\cH$ satisfies
\[
|\mathcal{T}_{v}| \leq6 \epsilon^{-2} \log(\epsilon^3 n) = o(|\tGC
|) .
\]
In particular, $|\mathcal{T}_{\rho}| = o(|\tGC|)$, and so (as it is
a tree) $\pi(\mathcal{T}_\rho) = o(1)$. However,~\eqref
{eq-bound-tau-rho} states that with probability
at least $\frac13$, the random walk started at some $w\in\mathcal
{T}_\rho$ does not escape from $\mathcal{T}_\rho$, hence
\[
\max_{w\in\tGC}\|\P_w(S_{2rs/3} \in\cdot) - \pi\|_{\mathrm{TV}}
\geq\frac{1}{4} ,
\]
where $\pi$ is the stationary measure for the random walk $S_t$ on
$\tGC$. In other words, we have that
\[
\tmix\bigl(\tfrac14 \bigr) \geq\tfrac23 rs = \tfrac1{96}\epsilon^{-3} \log
^2(\epsilon^3 n) ,
\]
as required.

\section{Mixing in the subcritical regime}\label{sec:mixing-subcritical}
In this section, we give the proof of Theorem \ref{mainthm-sub}. By
Theorem \ref{mainthm-super} and the well known duality between the
subcritical and supercritical regimes (see \cite{Luczak90}), it
suffices to establish the statement for the subcritical regime of $\cG(n,p)$.

For the upper bound, by results of \cite{Bollobas84} and \cite
{Luczak90} (see also \cite{NP07}), we know that the largest component
has size $O(\epsilon^{-2} \log(\epsilon^3 n))$ w.h.p., and by
results of \cite{Luczak98}, the largest diameter of a component is
w.h.p.\ $O(\epsilon^{-1} \log(\epsilon^3 n))$. Therefore, by the
commute time identity \eqref{eq-Ev-tau-w} the maximal hitting time to
a vertex is $O(\epsilon^{-3} \log^2(\epsilon^3 n))$ uniformly for
all components, and using the well-known fact that $\tmix= O(\max
_{x,y} \E_x \tau_y)$ (see, e.g., \cite{AF}, Chapter 2) we arrive at
the desired upper bound on the mixing time.

In order to establish the lower bound, we will demonstrate the
existence of a component with a certain structure, and show that the
order of the mixing time on this particular component matches the above
upper bound.

To find this component, we apply the usual exploration process until
$\epsilon n$ vertices are exposed. By definition, each component
revealed is a Galton--Watson tree (the exploration process does not
expose the tree-excess) where the offspring distribution is
stochastically dominated by $\Bin(n , \frac{1 - \epsilon}{n})$ and
stochastically dominates $\Bin(n , \frac{1 - 2\epsilon}{n})$.

It is well known [see, e.g., \cite{Lindvall}, equation (1.12)] that
for any $\lambda> 0$,
%
\[
\biggl\|\Bin\biggl(n, \frac{\lambda}{n}\biggr) - \Po(\lambda) \biggr\|_{\mathrm{TV}}
\leq\lambda^2/n .
\]
It follows that when discovering the first $\epsilon n$ vertices, we
can approximate the binomial variables by Poisson variables, at the
cost of a total error of at most $\epsilon n (1/n) = \epsilon= o(1)$.
%
\begin{lemma}\label{lem-number-components}
With high probability, once $\epsilon n$ vertices are exposed in the
exploration process, we will have discovered at least $\epsilon^2 n/2$
components.
\end{lemma}
\begin{pf}
Notice that each discovered component is stochastically dominated (with
respect to containment) by a $\operatorname{Poisson}(1-\epsilon
)$--Galton--Watson tree. Thus, the probability that the first $\epsilon
^2 n/2$ components contain more than $\epsilon n $ vertices is bounded
by the probability that the total size of $\epsilon^2 n/2$ independent
PGW($1-\epsilon$)-trees is larger than $\epsilon n$. The latter can be
estimated (using Chebyshev's inequality and Claim \ref
{claim-variance-PGW}) by
\[
\P\Biggl(\sum_{i=1}^{\epsilon^2 n/2}|\mathcal{T}_i| \geq
\epsilon n \Biggr) \leq\frac{\epsilon^2 n \epsilon^{-3}}{(\epsilon
n/2)^2} = 4(\epsilon^3 n)^{-1} = o(1) .
\]
\upqed
\end{pf}

For a rooted tree $\mathcal{T}$, we define the following event,
analogous to the event~$A_{r, s}(\mathcal{T})$ from Section \ref
{subsec:pgw}:
\[
B_{r, s}(\mathcal{T}) \deq\{\exists v, w \in\mathcal{T} \mbox{
such that }|\mathcal{T}_v| \geq s , |\mathcal{T}_w| \geq s \mbox{
and } \dist(v, w) = r\} .
\]
The next lemma estimates the probability that the above defined event
occurs in a PGW-tree.
\begin{lemma}\label{lem-prob-B}
Let $\mathcal{T}$ be a $\mathrm{PGW}(1-2\epsilon)$-tree and set
$r = \lceil\frac{1}{20} \epsilon^{-1} \log(\epsilon^3 n) \rceil$
and $s= \frac{1}{64} \epsilon^{-2} \log(\epsilon^3 n)$.
Then for some $c>0$ and any sufficiently large $n$,
\[
\P(B_{r, s}(\mathcal{T})) \geq c \epsilon(\epsilon^3 n)^{-1/2} .
\]
\end{lemma}
\begin{pf}
$\!\!$The proof follows the general argument of Lemma \ref
{lem-bush-diam-volume}. By Lem\-ma~\ref{lem-PGW-diameter},
\[
\P(L_{1/\epsilon} \neq\varnothing) \asymp\epsilon.
\]
Combined with the proof of Claim \ref{claim-variance-PGW} [see \eqref
{eq-var-L-i} in particular], we get that
\[
\E(|L_{1/\epsilon}| \mid L_{1/\epsilon}\neq\varnothing) \asymp
\epsilon^{-1} \quad  \mbox{and}\quad  \var(|L_{1/\epsilon}| \mid
L_{1/\epsilon} \neq\varnothing) \asymp\epsilon^{-2} .
\]
Applying Chebyshev's inequality, we get that for some constants $c_1, c_2>0$
\[
\P(|L_{1/\epsilon}| > c_1\epsilon^{-1} \mid L_{1/\epsilon} \neq
\varnothing) \geq c_2 .
\]
Repeating the arguments for the proof of Lemma \ref
{lem-bush-diam-volume}, we conclude that for a~PGW($1-2\epsilon$)-tree
$\mathcal{T}$, the probability that the event $A_{r, s}(\mathcal{T})$
occurs (using~$r, s$ as defined in the current lemma) is at least
$\epsilon(\epsilon^3 n)^{-1/4}$ for $n$ large enough. Thus [by the
independence of the subtrees rooted in the $(1/\epsilon)$th level],
\[
\P\biggl(\bigcup \{A_{r,s}(\mathcal{T}_u) \cap A_{r, s}(\mathcal
{T}_{u'})\dvtx u, u' \in L_{1/\epsilon} , u \neq u' \} \bigm|
|L_{1/\epsilon}| > c_1 \epsilon^{-1} \biggr) \geq c (\epsilon^3 n)^{-1/2}
\]
for some $c > 0$. Altogether, we conclude that for some $c' > 0$,
\[
\P\biggl(\bigcup\{A_{r,s}(\mathcal{T}_u) \cap A_{r, s}(\mathcal{T}_{u'})\dvtx
u, u' \in L_{1/\epsilon}, u \neq u' \} \biggr) \geq c' \epsilon
(\epsilon^3 n)^{-1/2} ,
\]
which immediately implies that required bound on $\P(B_{r, s}(\mathcal{T}))$.
\end{pf}

Combining Lemmas \ref{lem-number-components} and \ref{lem-prob-B}, we
conclude that w.h.p., during our exploration process we will find a
tree $\mathcal{T}$ which satisfies the event $B_{r,s}(\mathcal{T})$
for $r, s$ as defined in Lemma \ref{lem-prob-B}. Next, we will show
that the component of~$\mathcal{T}$ is indeed a tree, namely, it has
no tree-excess. Clearly, edges belonging to the tree-excess can only
appear between vertices that belong either to the same level or to
successive levels (the root of the tree $\mathcal{T}$ is defined to be
the vertex in $\mathcal{T}$ that is first exposed). Therefore, the
total number of candidates for such edges can be bounded by $4\sum_{i}
|L_i|^2$ where $L_i$ is the $i$th level of vertices in the tree. The
next claim provides an upper bound for this sum.
\begin{claim}
Let $r, s$ be defined as in Lemma \ref{lem-prob-B}. Then the $\mathrm
{PGW}(1-\epsilon)$-tree $\mathcal{T}$ satisfies $\E [\sum_i |L_i|^2
\mid B_{r, s}(\mathcal{T}) ] = O ( \epsilon^{-3} \sqrt{\epsilon^3
n} )$.
\end{claim}
\begin{pf}
Recalling Claim \ref{claim-variance-PGW} and in particular equation
\eqref{eq-var-L-i}, it follows that
$\E(\sum_i |L_i|^2 ) \leq\epsilon^{-2}$.
Lemma \ref{lem-prob-B} now implies the required upper bound.
\end{pf}

By the above claim and Markov's inequality, we deduce that w.h.p.\
there are, say, $O (\epsilon^{-3}(\epsilon^3 n)^{2/3} )$ candidates
for edges in the tree-excess of the component of~$\mathcal{T}$.
Crucially, whether or not these edges appear is independent of the
exploration process, hence the probability that any of them appears is
at most $O ((\epsilon^3 n)^{-1/3} ) = o(1)$.
Altogether, we may assume that the component of $\mathcal{T}$ is
indeed a tree which satisfies the event $B_{r, s}(\mathcal{T})$.

It remains to establish the lower bound on the mixing time of the
random walk on the tree $\mathcal{T}$. Let $v, w$ be two distinct
vertices in the $r$th level satisfying $|\mathcal{T}_v| \geq s$ and
$|\mathcal{T}_w| \geq s$. By the same arguments used to prove \eqref
{eq-bound-tau-rho}, we have that
\[
\max_{u\in\mathcal{T}_v} \P_u(\tau_w \geq10^{-3}rs)\geq1 -
10^{-3} .
\]
Recall that w.h.p.\ $|\mathcal{T}|\leq6\epsilon^{-2}\log(\epsilon
^3n) = 384 s$. It now follows that w.h.p.\ the mixing time of the
random walk on this components satisfies
\[
\tmix(\delta) \geq10^{-3} rs \qquad  \mbox{for } \delta= \tfrac
{1}{384} - 10^{-3} \geq10^{-3} .
\]
The lower bound on $\tmix(\frac14)$ now follows from the definition
of $r,s$.

\section*{Acknowledgments}
We thank Asaf Nachmias for helpful discussions at an early stage of
this project.


%

\printaddresses

\end{document}